\documentclass[12pt,reqno]{article}

\usepackage[usenames]{color}
\usepackage{amssymb}
\usepackage{amsmath}
\usepackage{amsthm}
\usepackage{amsfonts}
\usepackage{amscd}
\usepackage{graphicx}

\usepackage[colorlinks=true,
linkcolor=webgreen,
filecolor=webbrown,
citecolor=webgreen]{hyperref}

\definecolor{webgreen}{rgb}{0,.5,0}
\definecolor{webbrown}{rgb}{.6,0,0}

\usepackage{color}
\usepackage{fullpage}
\usepackage{float}

\usepackage{graphics}
\usepackage{latexsym}

\DeclareMathOperator{\Li}{Li}
\DeclareMathOperator{\Cl}{Cl}

\begin{document}

\theoremstyle{plain}
\newtheorem{theorem}{Theorem}
\newtheorem{corollary}[theorem]{Corollary}
\newtheorem{proposition}{Proposition}
\newtheorem{lemma}{Lemma}
\newtheorem{example}{Example}
\newtheorem*{remark}{Remark}

\begin{center}
\vskip 1cm
{\LARGE\bf Integration Formulas Involving\\ Fibonacci and Lucas Numbers }

\vskip 1cm

{\large
Kunle Adegoke \\
Department of Physics and Engineering Physics \\ Obafemi Awolowo University, 220005 Ile-Ife, Nigeria \\
\href{mailto:adegoke00@gmail.com}{\tt adegoke00@gmail.com}

\vskip 0.2 in

Robert Frontczak \\
Independent Researcher \\ Reutlingen,  Germany \\
\href{mailto:robert.frontczak@web.de}{\tt robert.frontczak@web.de}

}

\end{center}

\vskip .2 in

\begin{abstract}
We present a range of difficult integration formulas involving Fibonacci and Lucas numbers and trigonometric functions.
These formulas are often expressed in terms of special functions like the dilogarithm and Clausen's function.
We also prove complements of integral identities of Dilcher (2000) and Stewart (2022).
Many of our results are based on a fundamental lemma dealing with differentiation of complex-valued Fibonacci (Lucas) functions.
\end{abstract}



\bigskip

\section{Introduction}

In a recent paper from 2022, Stewart \cite{Stewart} derived some appealing integral representations for Fibonacci numbers 
$F_n$ and Lucas numbers $L_n$. For instance, he proved the representation \cite[Theorem 2.1]{Stewart}
\begin{equation}\label{eq.stewart}
\frac{F_{k n}}{F_k} = \frac{n}{2^n} \int_{-1}^1 \left (L_k + F_k x \sqrt{5}\right )^{n-1}\,dx \qquad (n,k\in\mathbb{N}).
\end{equation}
The special case of this identity for $k=1$ is also discussed in Stewart's paper \cite{Stewart2} from 2023.
Also, in 2015, Glasser and Zhou \cite{Glasser} worked out an explicit integral representation for $F_n$ involving trigonometric functions.
Indeed, the main result in their paper is the representation of the form
\begin{equation}
F_n = \frac{\alpha^n}{\sqrt{5}} - \frac{2}{\pi}\int_0^\infty \frac{\sin(x/2)}{x} \frac{\cos(nx)-2\sin(nx)\sin(x)}{5\sin^2(x)+\cos^2(x)} \,dx,
\end{equation}
where $\alpha=(1+\sqrt{5})/2$ is the golden ratio and $n\in\mathbb{N}_0$. 
Another representation is given by Andrica and Bagdasar in \cite{Andrica}. The last example for such representations comes from
the paper by Dilcher \cite{Dilcher} from 2000 where he showed (among others) that
\begin{equation}\label{eq.dilcher}
F_{2n} = \frac{n}{2} \left (\frac{3}{2}\right )^{n-1} \int_0^\pi \left ( 1+\frac{\sqrt{5}}{3} \cos(x)\right )^{n-1}\sin(x) \,dx.
\end{equation}

In this paper, we go in the same direction. However, we do not intend to prove explicit integral representations for Fibonacci and Lucas numbers, but instead we deal with integration formulas involving these sequences and combinations of trigonometric functions.
We begin by proving the following complements of Stewart's and Dilcher's integral identities
\begin{equation}\label{eq.bju5530}
\int_{- 1}^1 {\left( {L_k  + F_k x\sqrt 5 } \right)^{n - 2} \left( {F_k \sqrt 5  + L_k x} \right)\,dx}
= \frac{{2^n }}{{(n - 1)\sqrt 5 }}\left( {\frac{{L_{kn} }}{{F_k }} - \frac{{F_{kn} L_k }}{{nF_k^2 }}} \right),\quad n\not\in\{0,1\},
\end{equation}
\begin{equation}\label{eq.stewart_compl2}
\int_{-1}^1 \left( L_k + F_k x\sqrt 5 \right)^{n - 2} x\,dx 
= \frac{2^n}{(n-1)\sqrt{5}F_k}\left ( \frac{L_{(n-1)k}}{2}-\frac{F_{nk}}{nF_k}\right ).
\end{equation}
and
\begin{equation}\label{eq.djfhep4}
\begin{split}
&\int_0^\pi {\left( {1 + \frac{{\sqrt 5 }}{3}\cos x} \right)^{n - 1} \ln \left( {1 + \frac{{\sqrt 5 }}{3}\cos x} \right)\sin x \,dx} \\
&\qquad = \frac{6}{{n\sqrt 5 }}\left( {\frac{2}{3}} \right)^n L_{2n} \ln \alpha + \left( {-\frac{1}{n} + \ln \left( {\frac{2}{3}} \right)} \right)\left( {\frac{2}{3}} \right)^n \frac{3}{n}F_{2n},\quad n\in\mathbb Z^+.
\end{split}
\end{equation}

Then, we prove a range of difficult integral identities of which we chose the following ones as a showcase:
\begin{equation*}
\int_0^{\pi /2} {\frac{{\tan ^2 x}}{{1 + L_{2r} \tan ^2 x + \tan ^4 x}}\,dx} = \begin{cases}
 \dfrac{\pi }{2}\,\dfrac{1}{{F_r \sqrt 5 (F_r \sqrt 5  + 2)}}, & \text{if $r$ is odd}; \\
 \dfrac{\pi }{2}\,\dfrac{1}{{L_r (L_r  + 2)}}, & \text{if $r$ is even}; \\
 \end{cases}
\end{equation*}
\begin{equation*}
\int_0^\pi {\frac{{x\sin ^3 x}}{{\left( {4 + 5F_{2r}^2 \sin ^2 x} \right)^2 }}\,dx}
= - \frac{1}{10} \frac{\pi}{F_{4r}^2} + \frac{{2\pi \sqrt 5}}{25} \frac{L_{4r}}{F_{4r}^3} r \ln \alpha,
\end{equation*}
\begin{equation*}
\int_0^{\pi/2} \frac{x^2}{L_{r}^2 + 4 + 4L_{r} \cos(2x)}\,dx 
= \frac{1}{L_r^2-4}\left ( \frac{\pi^3}{24} + \frac{\pi}{2} \Li_2\left (\frac{2}{L_r}\right )\right ), \quad r\geq 2
\end{equation*}
and
\begin{equation*}
\begin{split}
&\int_0^\pi  {\frac{{x^2 \cos (3x)}}{{L_r^2  - 4\cos ^2 (2x)}}\,dx} \\
&\qquad = \left( {\frac{1}{L_r} - 1} \right)\frac{\pi }{2}\frac{{\sqrt { \beta ^r } }}{{1 - \beta ^r }}\left( {\Li_2 \left( {\sqrt { \beta ^r } } \right) - \Li_2 \left( { - \sqrt { \beta ^r } } \right)} \right)\\
&\qquad\quad+ \left( {\frac{1}{L_r} + 1} \right)\frac{\pi }{2}\frac{{\sqrt { \beta ^r } }}{{1 + \beta ^r }}\left( {\Cl_2 \left( {2\arctan \left( {\sqrt { \beta ^r } } \right)} \right) + \Cl_2 \left( {\pi  - 2\arctan \left( {\sqrt { \beta ^r } } \right)} \right)} \right)\\
&\qquad\qquad + \left( {\frac{1}{L_r} + 1} \right)\frac{{\pi \sqrt { \beta ^r } }}{{1 + \beta ^r }}\arctan \left( {\sqrt { \beta ^r } } \right)\ln \left( {\sqrt { \beta ^r } } \right). \qquad r\geq 2 \,\,\text{even}.
\end{split}
\end{equation*}

\medskip

Our paper is particularly inspired by the following identities of Lewin \cite{Lewin}:

\begin{align}
& \int_0^{\pi/2} \Li_2 ( - q^2 \tan ^2 x)\,dx = 2\pi \Li_2 (- q),\quad q\ge0, \label{eq.nx3w40i} \\
& \int_0^\infty \frac{{\arctan (qx)}}{{1 + x^2 }}\,dx = \frac{{\pi ^2 }}{8} - \frac{1}{2}\Li_2 \left( {\frac{{1 - q}}{{1 + q}}} \right)
+ \frac{1}{2}\Li_2 \left( { - \frac{{1 - q}}{{1 + q}}} \right), \label{eq.yjqd44q} \\
& \int_0^{\pi /2} \arctan (Q\csc x)\,dx = \frac{{\pi^2 }}{4} - \Li_2 \left( {\sqrt {1 + Q^2 } - Q} \right) +
\Li_2 \left( {- \sqrt {1 + Q^2 } + Q} \right), \label{eq.qk18ai6}\\
& \int_0^\pi x\arctan \left( {\frac{{2q}}{{1 - q^2 }}\sin x} \right)\,dx = \pi \Li_2 (q) - \pi \Li_2 (- q),\quad q^2 < 1 \label{eq.sjcljni}.\\
& \int_0^{\pi/2} \frac{{x^2 \,dx}}{{1 - Q\cos (2x)}} = \frac{{1 + q^2 }}{{1 - q^2 }}\left( {\frac{{\pi^3 }}{{24}}
+ \frac{\pi }{2}\Li_2 (- q)} \right), q^2  < 1,Q = \frac{{2q}}{{1 + q^2 }}, \label{eq_eleven} \\
& \int_0^\pi \frac{{x^2 }}{{1 - Q\cos ^2 x}}\,dx = \frac{{1 + q}}{{1 - q}}\left( {\frac{{\pi^3 }}{3} + \pi \Li_2 (q)} \right),
\quad q < 1,\quad Q = \frac{{4q}}{{(1 + q)^2 }},\label{eq.n71rwsq}
\end{align}
\begin{align}
& \int_0^\pi \frac{{x^2 \,dx}}{{1 - Q\cos (2x)}} = \frac{{1 + q^2 }}{{1 - q^2 }}\left( {\frac{{\pi ^3 }}{3} + \pi \Li_2 (q)} \right),
\quad q < 1,\quad Q = \frac{{2q}}{{1 + q^2 }},\label{eq.lh0gp48} \\
& \int_0^\pi \frac{{x^2 \cos x\,dx}}{{1 - Q\cos (2x)}} = - \pi \frac{{1 + q^2 }}{{1 - q}}\frac{{\Li_2 \left( {\sqrt q } \right)
- \Li_2 \left( { - \sqrt q } \right)}}{{\sqrt q }},q < 1, Q = \frac{{2q}}{{1 + q^2 }}\label{eq.b1e7nal}.
\end{align}

Obviously, the common feature in all these results is the appearance of the dilogarithm $\Li_2(z)$ on one or both sides of the equations.
This special function is defined by
\begin{equation*}
\Li_2(z) = \sum_{k=1}^\infty \frac{z^k}{k^2}, \qquad |z|<1.
\end{equation*}

We proceed with a definition of the Fibonacci numbers $F_n$ and the Lucas numbers $L_n$, and with some lemmas which we be used later. 
Both sequences are defined, for \text{$n\in\mathbb Z$}, through the recurrence relations $F_n = F_{n-1}+F_{n-2}, n\ge 2,$ 
with initial values $F_0=0, F_1=1$ and $L_n = L_{n-1}+L_{n-2}$ with $L_0=2, L_1=1$. For negative subscripts we have $F_{-n}=(-1)^{n-1}F_n$ 
and $L_{-n}=(-1)^n L_n$. They possess the explicit formulas (known as the Binet forms)
\begin{equation*}
F_n = \frac{\alpha^n - \beta^n }{\alpha - \beta },\quad L_n = \alpha^n + \beta^n,\quad n\in\mathbb Z,
\end{equation*}
with $\alpha=(1+\sqrt{5})/2$ and $\beta=(1-\sqrt{5})/2$.
For more information we refer to the books by Koshy \cite{Koshy} and Vajda \cite{Vajda}.

\begin{lemma}
If $z=2\arctan(\beta^r/i^r)$ where $r$ is an integer and $i$ is the imaginary unit, then
\begin{equation}\label{eq.j5m1ag9}
\cos z = \frac{F_r \sqrt{5}}{L_r},\qquad \sin z = \frac{2i^r}{L_r},\qquad \tan z = \frac{2i^r}{F_r \sqrt{5}}.
\end{equation}
\end{lemma}
\begin{proof}
This is a consequence of the fact that if $z=2\arctan(p/q)$, then
\begin{equation*}
\cos z = \frac{{q^2  - p^2 }}{{q^2  + p^2 }},\qquad \sin z = \frac{{2pq}}{{q^2  + p^2 }},\qquad \tan z = \frac{{2pq}}{{q^2 - p^2 }}.
\end{equation*}
So, for instance,
\begin{equation*}
\cos z = \frac{(-1)^r - \beta^{2r}}{(-1)^r + \beta^{2r}} = \frac{\alpha^r - \beta^{r}}{\alpha^r + \beta^{r}} = \frac{\sqrt{5}F_r}{L_r},
\end{equation*}
as $\alpha\beta=-1$. The remaining relations also follow immediately.
\end{proof}

\begin{lemma}\label{lem.fibfunctions}
Let $f(x)$ and $l(x)$ be the infinite times differentiable, complex-valued Fibonacci and Lucas functions defined by
\begin{equation}\label{eq.r280nsg}
f(x)=\frac{\alpha^x - \beta^x}{\alpha - \beta},\quad l(x)=\alpha^x + \beta^x,\quad x\in\mathbb R.
\end{equation}
Then
\begin{equation}\label{eq.nfgepkq}
\left. f(x) \right|_{x = j\in\mathbb Z} = F_j ,\quad\left. l(x) \right|_{x = j\in\mathbb Z} = L_j;
\end{equation}
and
\begin{equation}\label{eq.wyg8jr6}
\Re\left( {\left. {\frac{d}{{\,dx}}f(x)} \right|_{x = j\in\mathbb Z} } \right) = \frac{{L_j }}{{\sqrt 5 }}\,\ln \alpha ,\quad\Re\left( {\left. {\frac{d}{{\,dx}}l(x)} \right|_{x = j\in\mathbb Z} } \right) = F_j \sqrt 5\, \ln \alpha,
\end{equation}
\begin{equation}\label{eq.pagcd4g}
\Im\left( {\left. {\frac{d}{\,dx}f(x)} \right|_{x = j\in\mathbb Z} } \right) =  - \frac{{\pi \beta ^j }}{{\sqrt 5 }},\quad\Im\left( {\left. {\frac{d}{{\,dx}}l(x)} \right|_{x = j\in\mathbb Z} } \right) = \pi \beta ^j.
\end{equation}
\end{lemma}
\begin{proof}
First, since $\beta$ is negative, we write
\begin{equation*}
\beta ^x  = \left( { - \beta } \right)^x \exp \left( {i\pi \left( {2m + 1} \right)x} \right),\quad m\in\mathbb Z,
\end{equation*}
so that
\begin{equation*}
\frac{d}{{dx}}\beta ^x  = \beta ^x \left( {i\pi \left( {2m + 1} \right) + \ln \left( { - \beta } \right)} \right),\quad m\in\mathbb Z.
\end{equation*}
We have
\begin{equation*}
\begin{split}
\frac{d}{{dx}}f(x) &= \frac{1}{{\alpha  - \beta }}\left( {\frac{d}{{dx}}\alpha ^x  - \frac{d}{{dx}}\beta ^x } \right)\\
 &= \frac{1}{{\alpha  - \beta }}\left( {\alpha ^x \ln \alpha  - \beta ^x \ln \left( { - \beta } \right) - i\pi \left( {2m + 1} \right)\beta ^x } \right)\\
& = \frac{1}{{\alpha  - \beta }}\left( {\alpha ^x \ln \alpha  + \beta ^x \ln \alpha  - \beta ^x \ln \alpha  - \beta ^x \ln \left( { - \beta } \right) - i\pi \left( {2m + 1} \right)\beta ^x } \right)\\
 &= \frac{1}{{\alpha  - \beta }}\left( {\left( {\alpha ^x  + \beta ^x } \right)\ln \alpha  - \beta ^x \ln \left( { - \alpha \beta } \right) - i\pi \left( {2m + 1} \right)\beta ^x } \right)\\
&= \frac{1}{{\alpha  - \beta }}\left( {\left( {\alpha ^x  + \beta ^x } \right)\ln \alpha  - i\pi \left( {2m + 1} \right)\beta ^x } \right).
\end{split}
\end{equation*}
The first identity in~\eqref{eq.wyg8jr6} and the first identity in~\eqref{eq.pagcd4g} now follow upon taking real
and imaginary parts. For the imaginary part, we used the principal value, $m=0$.

The derivation of the second identity in~\eqref{eq.wyg8jr6} and the second identity in~\eqref{eq.pagcd4g} proceeds along the same line.
\end{proof}

\begin{lemma}
If $r$ is an integer, then
\begin{equation}\label{eq.s2rjcqb}
1 - \beta ^{2r}  =  \begin{cases}
 \beta ^r F_r \sqrt 5,&\text{$r$ even};  \\
  - \beta ^r L_r,&\text{$r$ odd};  \\
 \end{cases} ,\quad 1 + \beta ^{2r}  =  \begin{cases}
 \beta ^r L_r,&\text{$r$ even};  \\
  - \beta ^r F_r \sqrt 5,&\text{$r$ odd}.  \\
 \end{cases}
\end{equation}
\end{lemma}
\begin{proof}
Let $r$ be even. Then,
\begin{equation*}
1 - \beta^{2r} = (-1)^r - \beta^{2r} = \beta^r (\alpha^r - \beta^r) = \beta^r \sqrt{5} F_r.
\end{equation*}
The other cases are proved in exactly the same manner.
\end{proof}

\begin{lemma}
If $r$ is an integer, then
\begin{equation}\label{eq.sp0oo7c}
\begin{split}
& F_{2r}  - 1 =  \begin{cases}
 F_{r - 1} L_{r + 1},&\text{$r$ odd};  \\
 L_{r - 1} F_{r + 1},&\text{$r$ even};  \\
 \end{cases} ,\quad F_{2r + 1}  - 1 =  \begin{cases}
 L_r F_{r + 1},&\text{$r$ odd};  \\
 F_r L_{r + 1},&\text{$r$ even} ; \\
 \end{cases} ,\\
& \qquad\qquad\qquad L_{2r + 1}  - 1 =  \begin{cases}
 L_r L_{r + 1},&\text{$r$ odd};  \\
 5F_r F_{r + 1},&\text{$r$ even} . \\
 \end{cases}
\end{split}
\end{equation}
\end{lemma}
\begin{proof}
Apply the Binet forms for $F_n$ and $L_n$, respectively.
\end{proof}

\begin{lemma}\label{lem.oszn90v}
If $x>0$, then
\begin{equation}\label{eq.dudu2ha}
\Re \Li_2 (ix) = \frac{1}{4}\Li_2 ( - x^2 ) = \Re \Li_2 ( - ix),\quad\text{\cite[p.293, Identity (7)]{Lewin}},
\end{equation}
and
\begin{equation}\label{eq.j0jplsn}
\begin{split}
\Im \Li_2 (ix) &= \arctan x\ln x + \frac12\,\Cl_2 (2\arctan x) + \frac12\,\Cl_2 (\pi  - 2\arctan x)\\
&\qquad =  - \Im \Li_2 ( - ix);
\end{split}
\end{equation}
where $\Cl_2$ is Clausen's function defined by~\cite[p.291]{Lewin} :
\begin{equation*}
\Cl_2 (y) = \sum_{n = 1}^\infty  {\frac{{\sin (ny)}}{{n^2 }}}  =  - \int_0^y  {\ln |2\sin (\theta /2)|d\theta },
\end{equation*}
and having the functional relations
\begin{gather}
\Cl_2 (\pi  + \theta ) =  - \Cl_2 (\pi  - \theta ),\\
\Cl_2 (\theta ) =  - \Cl_2 (2\pi  - \theta ),\\
\frac{1}{2}\Cl_2 (2\theta ) = \Cl_2 (\theta ) - \Cl_2 (\pi  - \theta )\label{eq.rqfmoa7};
\end{gather}
with the special values
\begin{equation}
\Cl_2(n\pi)=0,\quad n\in\mathbb Z^+,
\end{equation}
and
\begin{equation}\label{eq.v1bi0dh}
\Cl_2(\pi/2)=G=-\Cl_2(3\pi/2),
\end{equation}
where $G=\sum_{j=0}^\infty (-1)^j / (1+2j)^2$ is Catalan's constant. For more information on these special functions see \cite{tric}.

\end{lemma}
Identity~\eqref{eq.j0jplsn} follows from~(see \cite[p.292, Identity (1)]{Lewin}) the fact that
\begin{equation*}
\Im \Li_2 (re^{iy} ) = \omega \ln r + \frac{1}{2}\Cl_2 (2\omega ) + \frac{1}{2}\Cl_2 (2y) - \frac{1}{2}\Cl_2 (2\omega  + 2y),
\end{equation*}
where
\begin{equation*}
\tan \omega  = \frac{{r\sin y}}{{1 - r\cos y}}.
\end{equation*}

\begin{lemma}[\cite{adegoke14}]\label{lem.z4tjirr}
If $s$ is a positive integer, then
\begin{equation}\label{eq.v5lli77}
\arctan (\beta ^s ) = \frac{1}{2}\arctan \left( {\frac{2}{{F_s \sqrt 5 }}} \right),\quad\text{if $s$ is even},
\end{equation}
and
\begin{equation}\label{eq.ka7hwfm}
\arctan (-\beta ^s ) = \frac{1}{2}\arctan \left( {\frac{2}{{L_s}}} \right),\quad\text{if $s$ is odd}.
\end{equation}
\end{lemma}
\section{Complements of the integral identities of Stewart and Dilcher}

To illustrate the importance and broad applicability of Lemma~\ref{lem.fibfunctions} we now 
derive~\eqref{eq.bju5530},~\eqref{eq.stewart_compl2} and~\eqref{eq.djfhep4}.

\begin{theorem}
For all integers $n\geq 2$ and $k\geq 1$ we have
\begin{equation*}\tag{\ref{eq.bju5530}}
\int_{ - 1}^1 {\left( {L_k + F_k x\sqrt 5 } \right)^{n - 2} \left( {F_k \sqrt 5  + L_k x} \right)\,dx}
= \frac{{2^n }}{{(n - 1)\sqrt 5 }}\left( {\frac{{L_{kn} }}{{F_k }} - \frac{{F_{kn} L_k }}{{nF_k^2 }}} \right)
\end{equation*}
and
\begin{equation}\label{eq.xsn0tmc}
\int_{ - 1}^1 {\left( {L_k + F_k x\sqrt 5 } \right)^{n - 2} (1 - x)\,dx} 
= \frac{{2^n }}{{(n - 1)F_k \sqrt 5 }}\left( { - \beta ^{(n - 1)k} + \frac{{F_{nk} }}{{nF_k }}} \right).
\end{equation}
\end{theorem}
\begin{proof}
The Fibonacci function form of~\eqref{eq.stewart} is
\begin{equation*}
\int_{- 1}^1 {\left( {l(t) + f(t)x\sqrt 5 } \right)^{n - 1}\,dx} = \frac{{f(tn)2^n }}{{nf(t)}},
\end{equation*}
which by differentiating with respect to $t$ gives
\begin{equation}\label{eq.tnsff40}
(n - 1)\int_{- 1}^1 {\left( {l(t) + f(t)x\sqrt 5 } \right)^{n - 2} \left( {\frac{d}{{dt}}l(t) + x\sqrt 5 \frac{d}{{dt}}f(t)} \right)}\,dx
= \frac{{2^n }}{{f(t)}}\frac{d}{{d t}}f(nt) - \frac{{2^n f(nt)}}{{nf(t)^2 }}\frac{d}{{dt}}f(t).
\end{equation}
Evaluating~\eqref{eq.tnsff40} at $t=k$ and taking real parts using~\eqref{eq.nfgepkq} and~\eqref{eq.wyg8jr6} and substituting
\begin{equation*}
\Re\left. {\frac{d}{{d t}}f(tn)} \right|_{t = k} = \frac{{L_{kn} }}{{\sqrt 5 }}\ln \alpha ,
\quad\Re\left. {\frac{d}{{dt}}l(t)} \right|_{t = k} = F_k \sqrt 5 \ln \alpha ,
\quad\Re\left. {\frac{d}{{dt}}f(t)} \right|_{t = k} = \frac{{L_k }}{{\sqrt 5 }}\ln \alpha,
\end{equation*}
\begin{equation*}
\left. {f(tn)} \right|_{t = k} = F_{kn} ,\quad\left. {f(t)} \right|_{t = k} = F_k ,\quad\left. {l(t)} \right|_{t = k} = L_k,
\end{equation*}
we obtain
\begin{equation*}
\begin{split}
&(n - 1)\int_{ - 1}^1 {\left( {L_k  + F_k x\sqrt 5 } \right)^{n - 2} \left( {F_k \sqrt 5 \ln \alpha  + x\sqrt 5 \frac{{L_k }}{{\sqrt 5 }}\ln \alpha } \right)\,dx} \\
&\qquad = \frac{{2^n }}{{F_k }}\frac{{L_{nk} }}{{\sqrt 5 }}\ln \alpha  - \frac{{2^n F_{nk} }}{{nF_k^2 }}\frac{{L_k }}{{\sqrt 5 }}\ln \alpha ,
\end{split}
\end{equation*}
from which~\eqref{eq.bju5530} follows.

Similarly, evaluating~\eqref{eq.tnsff40} at $t=k$ and taking imaginary parts using~\eqref{eq.nfgepkq} and~\eqref{eq.pagcd4g} and substituting
\begin{equation*}
\Im\left. {\frac{d}{{d t}}f(tn)} \right|_{t = k} = -\frac{\pi\beta^{nk}}{\sqrt 5} ,
\quad\Im\left. {\frac{d}{{dt}}l(t)} \right|_{t = k} = \pi\beta^k ,
\quad\Im\left. {\frac{d}{{dt}}f(t)} \right|_{t = k} = -\frac{\pi\beta^k}{\sqrt 5},
\end{equation*}
\begin{equation*}
\left. {f(tn)} \right|_{t = k} = F_{kn} ,\quad\left. {f(t)} \right|_{t = k} = F_k ,\quad\left. {l(t)} \right|_{t = k} = L_k,
\end{equation*}
we have
\begin{equation*}
\begin{split}
&(n - 1)\int_{ - 1}^1 {\left( {L_k  + F_k x\sqrt 5 } \right)^{n - 2} \left( {\pi \beta ^k  + x\sqrt 5 \left( { - \frac{{\pi \beta ^k }}{{\sqrt 5 }}} \right)} \right)\,dx} \\
&\qquad = \frac{{2^n }}{{F_k }}\left( { - \frac{{\pi \beta ^{nk} }}{{\sqrt 5 }}} \right) - \frac{{2^n F_{nk} }}{{nF_k^2 }}\left( { - \frac{{\pi \beta ^k }}{{\sqrt 5 }}} \right),
\end{split}
\end{equation*}
and hence~\eqref{eq.xsn0tmc} after dividing through by $\pi\beta^k$.
\end{proof}

\begin{corollary}
For all integers $n\geq 2$ and $k\geq 1$ we have
\begin{equation*}\tag{\ref{eq.stewart_compl2}}
\int_{-1}^1 \left( L_k + F_k x\sqrt 5 \right)^{n - 2} x\,dx 
= \frac{2^n}{(n-1)\sqrt{5}F_k}\left ( \frac{L_{(n-1)k}}{2}-\frac{F_{nk}}{nF_k}\right ).
\end{equation*}
\end{corollary}
\begin{proof}
Combine \eqref{eq.stewart} with \eqref{eq.xsn0tmc}.
\end{proof}

The complement of Dilcher's identity is given in the next theorem.
\begin{theorem}
For all integers $n\geq 1$, we have
\begin{equation*}\tag{\ref{eq.djfhep4}}
\begin{split}
&\int_0^\pi {\left( {1 + \frac{{\sqrt 5 }}{3}\cos x} \right)^{n - 1} \ln \left( {1 + \frac{{\sqrt 5 }}{3}\cos x} \right)\sin x\,dx} \\
&\qquad = \frac{6}{{n\sqrt 5 }}\left( {\frac{2}{3}} \right)^n L_{2n} \ln \alpha + \left( -\frac{1}{n} + \ln \left( {\frac{2}{3}} \right) \right)\left( {\frac{2}{3}} \right)^n \frac{3}{n}F_{2n}.
\end{split}
\end{equation*}
\end{theorem}
\begin{proof}
Differentiating the Fibonacci function form of~\eqref{eq.dilcher}, that is,
\begin{equation*}
\int_0^\pi \left( {1 + \frac{{\sqrt 5 }}{3}\cos x} \right)^{t - 1} \sin x\,dx = \frac{{2f(2t)}}{t}\left( {\frac{2}{3}} \right)^{t - 1}
\end{equation*}
with respect to $t$ gives
\begin{equation}\label{eq.q6rf1fk}
\begin{split}
&\int_0^\pi {\left( {1 + \frac{{\sqrt 5 }}{3}\cos x} \right)^{t - 1} \ln \left( {1 + \frac{{\sqrt 5 }}{3}\cos x} \right)\sin x\,dx} \\
&\qquad = 4\left( {\frac{2}{3}} \right)^{t - 1} \frac{d}{{dt}}f(2t) - 2\frac{{f(2t)}}{{t^2 }}\left( {\frac{2}{3}} \right)^{t - 1} + \frac{{2f(t)}}{t}\left( {\frac{2}{3}} \right)^{t - 1} \ln \left( {\frac{2}{3}} \right).
\end{split}
\end{equation}
Evaluating~\eqref{eq.q6rf1fk} at $t=n$ and taking the real part gives
\begin{equation*}
\begin{split}
&\int_0^\pi {\left( {1 + \frac{{\sqrt 5 }}{3}\cos x} \right)^{n - 1} \ln \left( {1 + \frac{{\sqrt 5 }}{3}\cos x} \right)\sin x\,dx} \\
&\qquad = \frac4n\left( {\frac{2}{3}} \right)^{n - 1} \Re\left. {\frac{d}{{dt}}f(2t)} \right|_{t = n} - 2\frac{{F_{2n} }}{{n^2 }}\left( {\frac{2}{3}} \right)^{n - 1} + \frac{{2F_{2n} }}{n}\left( {\frac{2}{3}} \right)^{n - 1} \ln \left( {\frac{2}{3}} \right)\\
&\qquad= \frac4n\left( {\frac{2}{3}} \right)^{n - 1} \frac{{L_{2n} }}{{\sqrt 5 }}\ln (\alpha ) - 2\frac{{F_{2n} }}{{n^2 }}\left( {\frac{2}{3}} \right)^{n - 1} + \frac{{2F_{2n} }}{n}\left( {\frac{2}{3}} \right)^{n - 1} \ln \left( {\frac{2}{3}} \right),
\end{split}
\end{equation*}
which simplifies to~\eqref{eq.djfhep4}.
\end{proof}

\section{Results associated with~\eqref{eq.nx3w40i}}

\begin{theorem}\label{thm_tan_id}
Let $r$ be an integer. Then
\begin{gather}
\int_0^{\pi /2} {\ln \left( {1 + L_{2r} \tan ^2 x + \tan ^4 x} \right)\,dx}  =  \begin{cases}
 \pi \ln (F_r \sqrt 5  + 2),&\text{if $r$ is odd}; \\
 \pi \ln (L_r  + 2), &\text{if $r$ is even};\label{eq.m6bi7ta}\\
 \end{cases} \\
 \int_0^{\pi /2} {\ln \left( {\frac{{\left( {1 + \alpha ^{2r}  + \tan ^2 x} \right)^2 }}{{1 + L_{2r} \tan ^2 x + \tan ^4 x}}} \right)\,dx}
	= \begin{cases}
 \pi r\ln \alpha,&\text{if $r$ is odd};  \\
 \pi \ln \left( {\frac{{(1 + \alpha ^r )^2 }}{{L_r  + 2}}} \right),&\text{if $r$ is odd}. \label{eq.k2xkue3}\\
 \end{cases}
\end{gather}
\end{theorem}
\begin{proof}
Differentiate~\eqref{eq.nx3w40i} with respect to $q$ to get
\begin{equation}\label{ln_tan_int_fromLew}
\int_0^{\pi/2} \ln \left (1+q^2 \tan^2 x \right ) \,dx = \pi \ln (1+q).
\end{equation}
Set $q=\alpha^r$ and $q=-\beta^r$, in turn, for the case when $r$ is an odd integer.
Use $q=\alpha^r$ and $q=\beta^r$, in turn, for the case when $r$ is an even integer.
Combine according to the Binet formulas; addition gives~\eqref{eq.m6bi7ta} while subtraction gives~\eqref{eq.k2xkue3}.
\end{proof}

\begin{corollary}\label{cor_tan_id}
If $r$ is an integer, then
\begin{equation}
\int_0^{\pi /2} {\frac{{\tan ^2 x}}{{1 + L_{2r} \tan ^2 x + \tan ^4 x}}\,dx}  =  \begin{cases}
 \dfrac{\pi }{2}\,\dfrac{1}{{F_r \sqrt 5 (F_r \sqrt 5  + 2)}},&\text{if $r$ is odd}; \\
 \dfrac{\pi }{2}\,\dfrac{1}{{L_r (L_r  + 2)}},&\text{if $r$ is even}. \\
 \end{cases}
\end{equation}
\end{corollary}
\begin{proof}
Differentiate the Fibonacci and Lucas function form of~\eqref{eq.m6bi7ta} with respect to $r$, making use of~\eqref{eq.wyg8jr6}.
\end{proof}

\begin{corollary}
If $r$ is an integer, then
\begin{equation}
\int_0^{\pi /2} \frac{1}{1 + L_{2r} \tan^2 x + \tan ^4 x}\,dx =  \begin{cases}
 \frac{\pi }{2}\,\frac{1}{L_{2r}(\sqrt{5} F_r + 2)}(L_{2r} + \sqrt{5} F_r - \frac{2}{\sqrt{5}F_r}),&\text{if $r$ is odd}; \\
 \frac{\pi }{2}\,\frac{1}{L_{2r}(L_r + 2)}(L_{2r} + L_r - \frac{2}{L_r}),&\text{if $r$ is even}. \\
 \end{cases}
\end{equation}
\end{corollary}
\begin{proof}
Replacing $q$ by $1/q$ in \eqref{ln_tan_int_fromLew} shows that
\begin{equation}
\int_0^{\pi/2} \ln \left (q^2 + \tan^2 x \right ) \,dx = \pi \ln (1+q).
\end{equation}
This yields
\begin{equation}\label{eq.t2k7wzu}
\int_0^{\pi/2} \frac{1}{q^2 + \tan^2 x} \,dx = \frac{\pi}{2 q (1+q)}.
\end{equation}
From here, we can proceed like in the proof of Theorem \ref{thm_tan_id} getting
\begin{equation*}
\int_0^{\pi/2} \frac{L_{2r}+2\tan^2 x}{1 + L_{2r} \tan^2 x + \tan ^4 x}\,dx = \frac{\pi}{2} \frac{\sqrt{5}F_r + L_{2r}}{\sqrt{5}F_r + 2}, \qquad r\,\mbox{odd}
\end{equation*}
and
\begin{equation*}
\int_0^{\pi/2} \frac{L_{2r}+2\tan^2 x}{1 + L_{2r} \tan^2 x + \tan ^4 x}\,dx = \frac{\pi}{2} \frac{L_r + L_{2r}}{L_r + 2}, \qquad r\,\mbox{even}.
\end{equation*}
This completes the proof.
\end{proof}

\begin{lemma}
If $n$ is a non-negative integer and $q$ is a positive number, then
\begin{equation}\label{eq.rokbvu0}
\int_0^{\pi /2} {\frac{{\,dx}}{{(q^2  + \tan ^2 x)^{n + 1}}}\sum_{k = 0}^{\left\lfloor {n/2} \right\rfloor } {\frac1{q^{2k}} \binom{{n}}{{2k}}\frac{{( - 1)^k }}{{2k + 1}}\tan ^{2k} x} }  = \frac{\pi }{2}\frac{1}{{n + 1}}\left( {\frac{1}{{q^{2n + 1} }} - \frac q{{(q(q + 1))^{n + 1} }}} \right).
\end{equation}
\end{lemma}
\begin{proof}
Differentiate~\eqref{eq.t2k7wzu} $n$ times.
\end{proof}

\begin{theorem}
If $n$ is a non-negative integer and $r$ is a positive integer, then
\begin{equation}
\begin{split}
&\int_0^{\pi /2} {\frac{{\,dx}}{{(L_r^2  + 5F_r^2 \tan ^2 x)^{n + 1} }}\sum_{k = 0}^{\left\lfloor {n/2} \right\rfloor } {\binom{{n}}{{2k}}\frac{{( - 1)^k }}{{2k + 1}}\left( {\frac{{5F_r^2 }}{{L_r^2 }}} \right)^k \tan ^{2k} x} }\\
&\qquad\qquad\qquad  = \frac{\pi }{2}\,\frac{1}{{n + 1}}\,\frac{1}{{L_r^n F_r \sqrt 5 }}\,\left( {\frac{1}{{L_r^{n + 1} }} - \frac{1}{{(2\alpha ^r )^{n + 1} }}} \right)
\end{split}
\end{equation}
and
\begin{equation}
\begin{split}
&\int_0^{\pi /2} {\frac{{\,dx}}{{(5F_r^2  + L_r^2 \tan ^2 x)^{n + 1} }}\sum_{k = 0}^{\left\lfloor {n/2} \right\rfloor } {\binom{{n}}{{2k}}\frac{{( - 1)^k }}{{2k + 1}}\left( {\frac{{L_r^2 }}{{5F_r^2 }}} \right)^k \tan ^{2k} x} }\\
&\qquad\qquad\qquad  = \frac{\pi }{2}\,\frac{1}{{n + 1}}\,\frac{1}{{(F_r\sqrt 5)^n L_r }}\,\left( {\frac{1}{{(F_r\sqrt 5)^{n + 1} }} - \frac{1}{{(2\alpha ^r )^{n + 1} }}} \right).
\end{split}
\end{equation}
\end{theorem}
\begin{proof}
Use $q=L_r/F_r \sqrt 5$ and $q=F_r \sqrt 5/L_r$ in~\eqref{eq.rokbvu0}.
\end{proof}
In particular, we mention the special cases
\begin{align}
\int_0^{\pi /2} \frac{\,dx}{(L_r^2 + 5F_r^2 \tan^2 x)^2} 
&= \frac{\pi}{4}\,\frac{1}{F_{2r}\sqrt{5}}\,\left(\frac{1}{L_r^{2}} - \frac{1}{4\alpha^{2r}} \right) \\
\int_0^{\pi /2} \frac{\,dx}{(5F_r^2 + L_r^2\tan^2 x)^2} 
&= \frac{\pi}{4}\,\frac{1}{F_{2r}\sqrt{5}}\,\left(\frac{1}{5 F_r^{2}} - \frac{1}{4\alpha^{2r}} \right)
\end{align}
with the special values
\begin{equation*}
\int_0^{\pi /2} \frac{\,dx}{(1 + 5 \tan^2 x)^2} = \frac{\pi\alpha}{16}
\end{equation*}
and
\begin{equation*}
\int_0^{\pi /2} \frac{\,dx}{(5 + \tan^2 x)^2} = \frac{\pi}{400}\left (2+\frac{7}{\alpha^2}\right ).
\end{equation*}

\begin{theorem}
If $n$ is a non-negative integer and $r$ is any integer, then
\begin{equation}
\begin{split}
&\int_0^{\pi /2} {\frac{{\,dx}}{{(1 + 3\tan ^2 x + \tan ^4 x)^{n + 1} }} }\\
&\qquad\qquad \times \left( {\sum_{k = 0}^{\left\lfloor {n/2} \right\rfloor } {\binom{{n}}{{2k}}\frac{{( - 1)^k \tan ^{2k} x}}{{2k + 1}}\sum_{j = 0}^{n + 1} {\binom{{n + 1}}{j}\tan ^{2j} xL_{2n + 2k - 2j + r + 2} } } } \right)
\\
&\qquad\qquad\qquad = \frac{\pi }{2}\frac{1}{{n + 1}}\left( {F_{2n + r + 1} \sqrt 5  - \alpha ^{r - 1}  + ( - 1)^{n + 1} \beta ^{3n + r + 2} } \right)
\end{split}
\end{equation}
and
\begin{equation}
\begin{split}
&\int_0^{\pi /2} {\frac{{\,dx}}{{(1 + 3\tan ^2 x + \tan ^4 x)^{n + 1} }} }\\
&\qquad\qquad \times \left( {\sum_{k = 0}^{\left\lfloor {n/2} \right\rfloor } {\binom{{n}}{{2k}}\frac{{( - 1)^k \tan ^{2k} x}}{{2k + 1}}\sum_{j = 0}^{n + 1} {\binom{{n + 1}}{j}\tan ^{2j} xF_{2n + 2k - 2j + r + 2} } } } \right)
\\
&\qquad\qquad\qquad = \frac{\pi }{2\sqrt 5}\frac{1}{{n + 1}}\left( {L_{2n + r + 1}  - \alpha ^{r - 1}  + ( - 1)^n \beta ^{3n + r + 2} } \right).
\end{split}
\end{equation}
\end{theorem}
\begin{proof}
Use $q=\alpha$ and $q=-\beta$ in~\eqref{eq.rokbvu0} and combine the resulting identities in accordance with the Binet formulas.
\end{proof}
In particular,
\begin{gather}
\int_0^{\pi /2} {\frac{{L_r \tan ^2 x + L_{r + 2} }}{{(1 + 3\tan ^2 x + \tan ^4 x)}}\,dx}  = \frac{\pi }{2}\left( {F_{r + 1} \sqrt 5  - \alpha ^{r - 1}  - \beta ^{r + 2} } \right),\\
\int_0^{\pi /2} {\frac{{F_r \tan ^2 x + F_{r + 2} }}{{(1 + 3\tan ^2 x + \tan ^4 x)}}\,dx}  = \frac{\pi }{2\sqrt 5}\left( {L_{r + 1}  - \alpha ^{r - 1}  + \beta ^{r + 2} } \right),
\end{gather}
with the special values
\begin{gather*}
\int_0^{\pi /2} {\frac{{1 - \tan ^2 x}}{{1 + 3\tan ^2 x + \tan ^4 x}}\,dx}  =  - \frac{\pi\beta ^3 }{2},\\
\int_0^{\pi /2} {\frac{\,dx}{{1 + 3\tan ^2 x + \tan ^4 x}}}  =   \frac{\pi\beta ^2 }{{\sqrt 5 }},\\
\int_0^{\pi /2} {\frac{{\tan ^2 x}}{{1 + 3\tan ^2 x + \tan ^4 x}}\,dx}  =  - \frac{\pi\beta ^3 }{{2\sqrt 5 }}.
\end{gather*}

\begin{remark}
Setting $q=\alpha^r$, $q=\pm\beta^r$ in~\eqref{eq.rokbvu0} will deliver more identities of this nature.
\end{remark}

\section{Results associated with~\eqref{eq.yjqd44q}}

\begin{theorem}
If $m$ is a non-negative integer and $r$ is a positive integer, then
\begin{equation}\label{eq.kj2w249}
\begin{split}
&\int_0^\infty  {\frac{{x^{2m + 1} \,dx}}{{(1 + x^2 )(1 + F_{2r} x^2 )^{m + 1}}}}=\int_0^\infty {\frac{{x \,dx}}{{(1 + x^2 )(F_{2r} + x^2 )^{m + 1}}}}\\
&\qquad=  \begin{cases}
  - \frac{1}{2}\sum_{j = 1}^m {\frac{1}{j}\frac{1}{{F_{2r}^j (F_{r - 1} L_{r + 1} )^{m - j + 1} }} + \frac{1}{2}\frac{{\ln F_{2r} }}{{(F_{r - 1} L_{r + 1} )^{m + 1} }}},&\text{if $r$ is odd, $r\ne 1$};  \\
  - \frac{1}{2}\sum_{j = 1}^m {\frac{1}{j}\frac{1}{{F_{2r}^j (L_{r - 1} F_{r + 1} )^{m - j + 1} }} + \frac{1}{2}\frac{{\ln F_{2r} }}{{(L_{r - 1} F_{r + 1} )^{m + 1} }}},&\text{if $r$ is even};  \\
 \end{cases}
\end{split}
\end{equation}

\begin{equation}\label{eq.dpbn6cy}
\begin{split}
&\int_0^\infty  {\frac{{x^{2m + 1} \,dx}}{{(1 + x^2 )(1 + F_{2r + 1} x^2 )^{m + 1}}}}=\int_0^\infty  {\frac{{x \,dx}}{{(1 + x^2 )(F_{2r + 1} + x^2 )^{m + 1}}}}\\
&\qquad=  \begin{cases}
  - \frac{1}{2}\sum_{j = 1}^m {\frac{1}{j}\frac{1}{{F_{2r + 1}^j (L_r F_{r + 1} )^{m - j + 1} }} + \frac{1}{2}\frac{{\ln F_{2r + 1} }}{{(L_r F_{r + 1} )^{m + 1} }}},&\text{if $r$ is odd};  \\
  - \frac{1}{2}\sum_{j = 1}^m {\frac{1}{j}\frac{1}{{F_{2r + 1}^j (F_r L_{r + 1} )^{m - j + 1} }} + \frac{1}{2}\frac{{\ln F_{2r + 1} }}{{(F_r L_{r + 1} )^{m + 1} }}},&\text{if $r$ is even};  \\
 \end{cases}
\end{split}
\end{equation}

\begin{equation}\label{eq.q2nviqw}
\begin{split}
&\int_0^\infty  {\frac{{x^{2m + 1} \,dx}}{{(1 + x^2 )(1 + L_{2r + 1} x^2 )^{m + 1}}}}=\int_0^\infty  {\frac{{x \,dx}}{{(1 + x^2 )(L_{2r + 1} + x^2 )^{m + 1}}}}\\
&\qquad=  \begin{cases}
  - \frac{1}{2}\sum_{j = 1}^m {\frac{1}{j}\frac{1}{{L_{2r + 1}^j (L_r L_{r + 1} )^{m - j + 1} }} + \frac{1}{2}\frac{{\ln L_{2r + 1} }}{{(L_r L_{r + 1} )^{m + 1} }}},&\text{if $r$ is odd};  \\
  - \frac{1}{2}\sum_{j = 1}^m {\frac{1}{j}\frac{1}{{L_{2r + 1}^j (5F_r F_{r + 1} )^{m - j + 1} }} + \frac{1}{2}\frac{{\ln L_{2r + 1} }}{{(5F_r F_{r + 1} )^{m + 1} }}},&\text{if $r$ is even};  \\
 \end{cases}
\end{split}
\end{equation}
\end{theorem}
\begin{proof}
Differentiating~\eqref{eq.yjqd44q} with respect to $q$ gives
\begin{equation*}
\int_0^\infty \frac{{x\,dx}}{{(1 + x^2 )(1 + qx^2 )}} = - \frac{1}{2}\frac{{\ln q}}{{1 - q}}
\end{equation*}
which writing $1/q$ for $q$ also means
\begin{equation*}
\int_0^\infty \frac{{x\,dx}}{{(1 + x^2 )(q + x^2 )}} = - \frac{1}{2}\frac{{\ln q}}{{1 - q}};
\end{equation*}
so that
\begin{equation}\label{eq.hqmqa9h}
\int_0^\infty \frac{{x\,dx}}{{(1 + x^2 )(1 + qx^2 )}} = - \frac{1}{2}\frac{{\ln q}}{{1 - q}} = \int_0^\infty \frac{{x\,dx}}{{(1 + x^2 )(q + x^2 )}}.
\end{equation}
Differentiating~\eqref{eq.hqmqa9h} $m$ times with respect to $q$ gives
\begin{equation}\label{eq.pdjjqgd}
\begin{split}
&\int_0^\infty  {\frac{{x^{2m + 1} \,dx}}{{(1 + x^2 )(1 + qx^2 )^{m + 1} }}}  = \int_0^\infty  {\frac{{x\,dx}}{{(1 + x^2 )(q + x^2 )^{m + 1} }}} \\
&\qquad = \frac{1}{2}\sum\limits_{j = 1}^m {\frac{{( - 1)^{m - j} }}{j}\frac{1}{{q^j (1 - q)^{m - j + 1} }} + \frac{{( - 1)^{m - 1} }}{2}\frac{{\ln q}}{{(1 - q)^{m + 1} }}}.
\end{split}
\end{equation}
Using $q=F_{2r}$, $q=F_{2r + 1}$ and $q=L_{2r + 1}$ in turn in~\eqref{eq.pdjjqgd} while making use of~\eqref{eq.sp0oo7c} produces~\eqref{eq.kj2w249},~\eqref{eq.dpbn6cy} and~\eqref{eq.q2nviqw}.
\end{proof}

\begin{theorem}
If $r$ is a non-zero integer, then
\begin{align}
&\int_0^\infty \frac{x^{2m + 1} \,dx}{(1 + x^2 )(L_r^2 + 5F_r^2 x^2 )^{m + 1}} = \int_0^\infty \frac{x\,dx}{(1 + x^2 )(5F_r^2 + L_r^2 x^2 )^{m + 1} } \nonumber \\
&\qquad = \frac{1}{2^{2m+3}}\left ( \sum_{j=1}^m \frac{(-1)^{r+(r+1)(m-j)}}{j} \left (\frac{4}{5 F_r}\right )^{j}
+ (- 1)^{(r+1)(m+1)} \ln \left( \frac{5F_r^2}{L_r^2} \right) \right ).
\end{align}

\end{theorem}
\begin{proof}
Set $q=5F_r^2/L_r^2$ in~\eqref{eq.pdjjqgd} and use the identity $L_n^2=5F_n^2+(-1)^n 4$.
\end{proof}

\begin{theorem}
If $r$ is a non-zero even integer, then
\begin{align}
&\int_0^\infty \frac{x^{2m + 1} \,dx}{(1 + x^2)(L_r^2 + 4 x^2 )^{m + 1}}
= \int_0^\infty \frac{x\,dx}{(1 + x^2 )(4 + L_r^2 x^2 )^{m + 1}} \nonumber \\
&\qquad = \frac{1}{2(5F_r^2)^{m+1}} \left ( \sum_{j=1}^m \frac{(-1)^{m-j}}{j} \left (\frac{5 F_r^{2}}{4}\right )^j
+ (-1)^{m - 1} \ln \left( \frac{4}{L_r^2} \right) \right ).
\end{align}
\end{theorem}
\begin{proof}
Set $q=4/L_r^2$ in~\eqref{eq.pdjjqgd} and use the identity $L_n^2=5F_n^2+(-1)^n 4$.
\end{proof}

\begin{theorem}
If $r$ is a positive odd integer, then
\begin{align}
&\int_0^\infty \frac{x^{2m + 1} \,dx}{(1 + x^2)(5 F_r^2 + 4 x^2 )^{m + 1}}
= \int_0^\infty \frac{x\,dx}{(1 + x^2 )(4 + 5 F_r^2 x^2 )^{m + 1}} \nonumber \\
&\qquad = \frac{1}{2 L_r^{2(m+1)}} \left ( \sum_{j=1}^m \frac{(-1)^{m-j}}{j} \left (\frac{L_r^{2}}{4}\right )^j
+ (-1)^{m - 1} \ln \left( \frac{4}{5 F_r^2} \right) \right ).
\end{align}
\end{theorem}
\begin{proof}
Set $q=4/(5F_r^2)$ in~\eqref{eq.pdjjqgd} and use the identity $L_n^2=5F_n^2+(-1)^n 4$.
\end{proof}

\begin{theorem}
If $r$ is a positive integer, then
\begin{align}
&\int_0^\infty \frac{x^{2m + 1} \,dx}{(1 + x^2)(1 + F_{4r+1} x^2 )^{m + 1}}
= \int_0^\infty \frac{x\,dx}{(1 + x^2 )(F_{4r+1} + x^2 )^{m + 1}} \nonumber \\
&\qquad = \frac{1}{2(F_{2r}L_{2r+1})^{m+1}} \left ( \ln F_{4r+1} - \sum_{j=1}^m \frac{1}{j} \left (\frac{F_{2r}L_{2r+1}}{F_{4r+1}}\right )^j\right ).
\end{align}
\end{theorem}
\begin{proof}
Set $q=F_{4r+1}$ in~\eqref{eq.pdjjqgd} and use the identity $F_{4n+1}-1=F_{2n}L_{2n+1}$.
\end{proof}

\section{Results associated with~\eqref{eq.qk18ai6}}

\begin{lemma}
If $0<q\le 1$ then
\begin{equation}\label{eq.cl66rls}
\begin{split}
\int_0^{\pi/2}  {\ln \left( {\frac{{\dfrac{{1 + q^2 }}{{2q}} + \sin x}}{{\dfrac{{1 + q^2 }}{{2q}} - \sin x}}} \right)\,dx}  
&= 2\Cl_2 \left( {2\arctan q} \right) + 2\Cl_2 \left( {\pi  - 2\arctan q} \right)\\
&\qquad\qquad + 4\arctan q\ln q.
\end{split}
\end{equation}
\end{lemma}
\begin{proof}
Replace $q$ by $i(1+q^2)/(2q)$ in~\eqref{eq.qk18ai6} and take the real part.
\end{proof}

\begin{theorem}\label{thm.j7hzxma}
If $r$ is an even integer, then
\begin{equation}
\begin{split}
\int_0^{\pi/2}  {\ln \left( {\frac{{L_r  + 2\sin x}}{{L_r  - 2\sin x}}} \right)\,dx}  &= 2\Cl_2 \left( {\arctan \left( {\frac{2}{{F_r \sqrt 5 }}} \right)} \right) + 2\Cl_2 \left( {\pi  - \arctan \left( {\frac{2}{{F_r \sqrt 5 }}} \right)} \right)\\
&\qquad - 2r\arctan \left( {\frac{2}{{F_r \sqrt 5 }}} \right)\ln \alpha 
\end{split}
\end{equation}
while if $r$ is an odd integer, then
\begin{equation}
\begin{split}
\int_0^{\pi/2}  {\ln \left( {\frac{{F_r \sqrt 5  + 2\sin x}}{{F_r \sqrt 5  - 2\sin x}}} \right)\,dx}  &= 2\Cl_2 \left( {\arctan \left( {\frac{2}{{L_r }}} \right)} \right) + 2\Cl_2 \left( {\pi  - \arctan \left( {\frac{2}{{L_r }}} \right)} \right)\\
&\qquad - 2r\arctan \left( {\frac{2}{{L_r }}} \right)\ln \alpha .
\end{split}
\end{equation}
\end{theorem}
\begin{proof}
Consider $r$ an even integer. Set $q=\beta^r$ in~\eqref{eq.cl66rls} and use~\eqref{eq.s2rjcqb},~\eqref{eq.j0jplsn} and~\eqref{eq.v5lli77}. Consider $r$ an odd integer. Set $q=-\beta^r$ in~\eqref{eq.cl66rls} and use~\eqref{eq.s2rjcqb},~\eqref{eq.j0jplsn} and~\eqref{eq.ka7hwfm}.
\end{proof}
In particular,
\begin{equation}
\int_0^{\pi/2} \ln \left( {\frac{{1 + \sin x}}{{1 - \sin x}}} \right)\,dx = 4G,
\end{equation}
which can be compared to other integral representations of $G$ like
\begin{equation*}
G = - \int_0^1 \frac{\ln x}{1+x^2}\,dx = \int_0^1 \frac{\tan^{-1} x}{x}\,dx.
\end{equation*}

Differentiating~\eqref{eq.qk18ai6} gives
\begin{equation}\label{eq.tkzy7wr}
\int_0^{\pi /2} \frac{{\sin x}}{{\sin ^2 x + Q^2 }}\,dx 
= \frac{1}{{\sqrt {1 + Q^2} }}\ln \left( {\frac{{1 - Q + \sqrt {1 + Q^2 } }}{{1 + Q - \sqrt {1 + Q^2 } }}} \right).
\end{equation}

\begin{theorem}
If $r$ is a non-zero integer, then
\begin{equation}\label{Fib_sin_id1}
\int_0^{\pi/2} \frac{\sin x}{5F_r^2 + L_r^2 \sin^2 x}\,dx = \frac{\sqrt{2}}{2 L_r \sqrt{L_{2r}}}
\ln \left( \frac{\sqrt{2}\beta^r + \sqrt{L_{2r}}}{\sqrt{2}\alpha^r - \sqrt{L_{2r}}} \right)
\end{equation}
and
\begin{equation}\label{Fib_sin_id2}
\int_0^{\pi/2} \frac{\sin x}{L_r^2 + 5F_r^2 \sin^2 x}\,dx = \frac{\sqrt{10}}{10 F_r \sqrt{L_{2r}}}
\ln \left( \frac{-\sqrt{2}\beta^r + \sqrt{L_{2r}}}{\sqrt{2}\alpha^r - \sqrt{L_{2r}}} \right).
\end{equation}
\end{theorem}
\begin{proof}
Set $Q=F_r\sqrt 5/L_r$ and $Q=L_r/(F_r\sqrt 5)$ in~\eqref{eq.tkzy7wr} and simplify making use of $L_n^2+5F_n^2=2L_{2n}$.
\end{proof}

\begin{theorem}
If $r\geq 2 \,\,(r\geq 1)$ is an integer, then
\begin{equation}\label{Fib_sin_id3}
\int_0^{\pi/2} \frac{\sin x}{F_{r-1}^2 + F_r^2 \sin^2 x}\,dx = \frac{1}{F_r \sqrt{F_{2r-1}}}
\ln \left( \frac{F_{r-2} + \sqrt{F_{2r-1}}}{F_{r+1} - \sqrt{F_{2r-1}}} \right)
\end{equation}
and
\begin{equation}\label{Fib_sin_id4}
\int_0^{\pi/2} \frac{\sin x}{L_{r-1}^2 + L_r^2 \sin^2 x}\,dx = \frac{1}{L_r \sqrt{5 F_{2r-1}}}
\ln \left( \frac{L_{r-2} + \sqrt{5 F_{2r-1}}}{L_{r+1} - \sqrt{5 F_{2r-1}}} \right).
\end{equation}
\end{theorem}
\begin{proof}
Set $Q=F_{r-1}/F_r$ and $Q=L_{r-1}/L_r$ in~\eqref{eq.tkzy7wr} and simplify making use of the Catalan identity.
\end{proof}

\begin{remark}
We mention that identities \eqref{Fib_sin_id3} and \eqref{Fib_sin_id4} can be generalized further.
For instance, we record that for each $k\geq 1$ and each odd $r\geq 1$ we have
\begin{equation}\label{Fib_sin_id5_gen}
\int_0^{\pi/2} \frac{\sin x}{F_{k}^2 + F_{k+r}^2 \sin^2 x}\,dx = \frac{1}{F_{k+r} \sqrt{F_r F_{2k+r}}}
\ln \left( \frac{F_{k+r} - F_k + \sqrt{F_r F_{2k+r}}}{F_{k+r} + F_k - \sqrt{F_r F_{2k+r}}} \right),
\end{equation}
which contains \eqref{Fib_sin_id3} and
\begin{equation}
\int_0^{\pi/2} \frac{\sin x}{F_{r}^2 + F_{2r}^2 \sin^2 x}\,dx = \frac{1}{F_{2r} \sqrt{F_r F_{3r}}}
\ln \left( \frac{F_{2r} - F_r + \sqrt{F_r F_{3r}}}{F_{2r} + F_r - \sqrt{F_r F_{3r}}} \right)
\end{equation}
as special cases.
\end{remark}

\begin{theorem}
If $r$ is a positive odd integer, then
\begin{equation}\label{Fib_sin_id6}
\int_0^{\pi/2} \frac{\sin x}{4 + L_r^2 \sin^2 x}\,dx = \frac{r}{\sqrt{5} F_{2r}} \ln \alpha.
\end{equation}
\end{theorem}
\begin{proof}
Set $Q=2/L_r$ in~\eqref{eq.tkzy7wr} and simplify.
\end{proof}

\begin{corollary}
If $r$ is a positive odd integer, then
\begin{equation}\label{Fib_sin_id7}
\int_0^{\pi/2} \frac{\sin^3 x}{(4 + L_r^2 \sin^2 x)^2}\,dx = \frac{1}{10F_{2r}^2}\left (\frac{2r L_{2r}}{\sqrt{5} F_{2r}} \ln \alpha - 1 \right).
\end{equation}
\end{corollary}
\begin{proof}
Differentiate the Fibonacci and Lucas function forms of~\eqref{Fib_sin_id6} and take the real part, using~\eqref{eq.wyg8jr6}.
\end{proof}

\begin{remark}
Noting that 
\begin{equation*}
\int_0^{\pi/2} \frac{{\sin x}}{{\sin ^2 x + Q^2 }}\,dx = \frac{1}{\sqrt {1 + Q^2 }} \ln \left( \frac{{1 + \sqrt {1 + Q^2 } }}{Q} \right).
\end{equation*}
it is clear the results presented in this section can be stated in a slightly different form. 
\end{remark}

\section{Results associated with~\eqref{eq.sjcljni}}

We can write~\eqref{eq.sjcljni} as
\begin{equation}\label{eq.hjxqrgq}
\int_0^\pi  {x\arctan (Q\sin x)\,dx}  = \pi \Li_2(q) - \pi \Li_2(-q),
\end{equation}
where
\begin{equation}
Q=Q(q)=\frac{2q}{1 - q^2};
\end{equation}
so that
\begin{equation}
\frac{{dQ}}{{dq}} = \frac{{2(1 + q^2 )}}{{(1 - q^2 )^2 }} = \frac{1}{{(1 - q)^2 }} + \frac{1}{{(1 + q)^2 }}.
\end{equation}
Differentiating~\eqref{eq.hjxqrgq} with respect to $q$, we have
\begin{equation}
\int_0^\pi  {\frac{{x\sin x}}{{1 + Q^2 \sin ^2 x}}\,dx}  = \frac{{\pi \ln \left( {\left| {\dfrac{{1 + q}}{{1 - q}}} \right|} \right)}}{{q\dfrac{{dQ}}{{dq}}}};
\end{equation}
that is
\begin{equation}\label{eq.u6jqlay}
\int_0^\pi  {\frac{{x\sin x}}{{1 + \left( {\dfrac{{2q}}{{1 - q^2 }}} \right)^2 \sin ^2 x}}\,dx}  = \frac{\pi }{2}\frac{{\ln \left( {\left| {\dfrac{{1 + q}}{{1 - q}}} \right|} \right)}}{{\dfrac{q}{{1 - q^2 }}\dfrac{{1 + q^2 }}{{1 - q^2 }}}}.
\end{equation}
We now proceed to derive from~\eqref{eq.u6jqlay} a couple of identities involving Fibonacci and Lucas numbers.

\begin{theorem}
If $r$ is a non-zero integer, then
\begin{gather}
\int_0^\pi  {\frac{{x\sin x}}{{L_r^2  + 4\sin ^2 x}}\,dx}  = \frac{\pi }{{2F_r \sqrt 5 }}\ln \left( {\frac{{F_r \sqrt 5  + 2}}{{L_r }}} \right),\quad\text{$r$ odd};\label{eq.cpwmq60}\\
\int_0^\pi  {\frac{{x\sin x}}{{L_r^2  - 4\sin ^2 x}}\,dx}  = \frac{\pi }{{2L_r }}\ln \left( {\frac{F_r \sqrt 5}{{L_r -2 }}} \right),\quad\text{$r$ even}.\label{eq.rdjra1d}
\end{gather}
\end{theorem}
\begin{proof}
Set $q=\beta^r$ in~\eqref{eq.u6jqlay} and use~\eqref{eq.s2rjcqb}.
\end{proof}

\begin{corollary}
If $r$ is a non-zero integer, then
\begin{gather}
\int_0^\pi  {\frac{{x\sin x}}{{\left( {L_r^2  + 4\sin ^2 x} \right)^2 }}\,dx}  = \frac{\pi}{{20\sqrt 5 F_r^3 }}\ln \left( {\frac{{F_r \sqrt 5  + 2}}{{L_r }}} \right) + \frac{\pi }{{10F_{2r}^2 }},\quad\text{$r$ odd};\\
\int_0^\pi  {\frac{{x\sin x}}{{\left( {L_r^2  - 4\cos ^2 x} \right)^2 }}\,dx}  = \frac{\pi}{{4 L_r^3 }}\ln \left( {\frac{{F_r \sqrt 5}}{{L_r - 2}}} \right) + \frac{\pi }{{10F_{2r}^2 }},\quad\text{$r$ even}.
\end{gather}
\end{corollary}
\begin{proof}
Differentiate the Fibonacci and Lucas function forms of~\eqref{eq.cpwmq60} and~\eqref{eq.rdjra1d} and take the real part, 
using~\eqref{eq.wyg8jr6}.
\end{proof}

\begin{theorem}
If $r$ is a non-negative integer, then
\begin{equation}\label{eq.rljj8to}
\int_0^\pi  {\frac{{x\sin x}}{{4 + 5F_{2r}^2 \sin ^2 x}}\,dx}  = \frac{{2\pi r\sqrt 5 }}{{5F_{4r} }}\ln \alpha.
\end{equation}
\end{theorem}
\begin{proof}
Set $q=F_r\sqrt 5/L_r$ in~\eqref{eq.u6jqlay}.
\end{proof}

\begin{corollary}
If $r$ is a non-negative integer, then
\begin{equation}
\int_0^\pi  {\frac{{x\sin ^3 x}}{{\left( {4 + 5F_{2r}^2 \sin ^2 x} \right)^2 }}\,dx}  =  - \frac{1}{{10}}\frac{\pi }{{F_{4r}^2 }} + \frac{{2\pi \sqrt 5 }}{25}\frac{{L_{4r} }}{{F_{4r}^3 }}r\ln \alpha .
\end{equation}
\end{corollary}
\begin{proof}
Differentiate the Fibonacci-Lucas function form of~\eqref{eq.rljj8to} and take the real part, using~\eqref{eq.wyg8jr6}.
\end{proof}

Next write~\eqref{eq.sjcljni} as
\begin{equation}
\int_0^\pi  {x\arctan (Q\sin x)\,dx}  = \pi \Li_2 \left( {\frac{{\sqrt {1 + Q^2 }  - 1}}{Q}} \right) - \pi \Li_2 \left( {\frac{{ - \sqrt {1 + Q^2 }  + 1}}{Q}} \right),\quad Q\in\mathbb R,
\end{equation}
which by writing $iQ$ for $Q$ also implies
\begin{equation}\label{eq.slw99if}
\begin{split}
&\int_0^\pi  {x\ln \left( {\frac{{1 + Q\sin x}}{{1 - Q\sin x}}} \right)\,dx}\\
&\qquad  = 2i\pi \Li_2 \left( {i\frac{{\left( {\sqrt {1 - Q^2 }  - 1} \right)}}{Q}} \right) - 2i\pi \Li_2 \left( { - i\frac{{\left( {\sqrt {1 - Q^2 }  - 1} \right)}}{Q}} \right),\quad Q^2<1,
\end{split}
\end{equation}
and which upon differentiation gives
\begin{equation}\label{eq.gjneyfk}
\int_0^\pi  {\frac{{x\sin x}}{{1 + Q^2 \sin ^2 x}}\,dx}  = \frac{\pi }{{Q\sqrt {1 + Q^2 } }}\ln \left( Q + \sqrt{1 + Q^2} \right),\quad Q\in\mathbb R.
\end{equation}
\begin{remark}
By setting $Q=2/L_r$ and $Q=2/F_r\sqrt 5$, in turn, in~\eqref{eq.slw99if}, similar results to those in Theorem~\ref{thm.j7hzxma} can be derived.
\end{remark}

\begin{theorem}
If $r$ is a non-zero integer, then
\begin{gather}
\int_0^\pi  {\frac{{x\sin x}}{{2L_{2r}  - L_r^2 \cos ^2 x}}\,dx}  = \frac{{\pi \sqrt 2 }}{{2L_r \sqrt {L_{2r} } }}\ln \left( {\frac{{\beta ^r \sqrt 2  + \sqrt {L_{2r} } }}{{\alpha ^r \sqrt 2  - \sqrt {L_{2r} } }}} \right),\label{eq.sc3t62n}\\
\int_0^\pi  {\frac{{x\sin x}}{{2L_{2r}  - 5F_r^2 \cos ^2 x}}\,dx}  = \frac{{\pi \sqrt 5 \sqrt 2 }}{{10F_r \sqrt {L_{2r} } }}\ln \left( {\frac{{ - \beta ^r \sqrt 2  + \sqrt {L_{2r} } }}{{\alpha ^r \sqrt 2  - \sqrt {L_{2r} } }}} \right)\label{eq.qo33h5m}.
\end{gather}
\end{theorem}
\begin{proof}
Set $Q=L_r/(F_r\sqrt 5)$ in~\eqref{eq.gjneyfk} to obtain~\eqref{eq.sc3t62n} and $Q=F_r\sqrt 5/L_r$ to obtain~\eqref{eq.qo33h5m}.
\end{proof}
Writing $iQ$ for $Q$ in~\eqref{eq.gjneyfk}, we have
\begin{equation}\label{eq.pglrqhp}
\int_0^\pi  {\frac{{x\sin x}}{{1 - Q^2 \sin ^2 x}}\,dx}  =   \frac{\pi }{{Q\sqrt {1 - Q^2 } }}\arctan \left( {\frac{{Q}}{{\sqrt {1 - Q^2 }}}} \right),\quad Q^2<1.
\end{equation}
\begin{theorem}\label{thm.fmk6krx}
If $r$ is a non-zero even integer, then
\begin{equation}
\int_0^\pi  {\frac{{x\sin x}}{{L_r^2  - 4\sin ^2 x}}\,dx}  = \frac{\pi }{2}\frac{1}{{F_r \sqrt 5 }}\arctan \left( {\frac{2}{{F_r \sqrt 5 }}} \right),
\end{equation}
while if $r$ is an odd integer, then
\begin{equation}
\int_0^\pi  {\frac{{x\sin x}}{{5F_r^2  - 4\sin ^2 x}}\,dx}  = \frac{\pi }{2}\frac{1}{{L_r }}\arctan \left( {\frac{2}{{L_r }}} \right).
\end{equation}
\end{theorem}
\begin{proof}
Set $Q=2/L_r$ and $Q=2/F_r \sqrt5$, in turn, in~\eqref{eq.pglrqhp} .
\end{proof}
In particular,
\begin{equation}
\int_0^\pi  {\frac{{x\sin x}}{{5 - 4\sin ^2 x}}\,dx}  = \frac{\pi }{2}\arctan 2.
\end{equation}
\begin{corollary}
If $r$ is a non-zero even integer, then
\begin{equation}
\int_0^\pi  {\frac{{x\sin x}}{{\left( {L_r^2 - 4\sin ^2 x} \right)^2 }}\,dx} = \frac{\pi }{{20\sqrt 5 F_r^3 }}\arctan \left( {\frac{2}{{F_r \sqrt 5 }}} \right) + \frac{\pi }{{10F_{2r}^2 }},
\end{equation}
while if $r$ is an odd integer, then
\begin{equation}
\int_0^\pi  {\frac{{x\sin x}}{{\left( {5F_r^2 - 4\sin ^2 x} \right)^2 }}\,dx} = \frac{\pi }{{4L_r^3 }}\arctan \left( {\frac{2}{{L_r }}} \right) + \frac{\pi }{{10F_{2r}^2 }}.
\end{equation}
\end{corollary}
\begin{proof}
Differentiate the Fibonacci-Lucas function forms of the identities in Theorem~\ref{thm.fmk6krx}.
\end{proof}

\begin{remark}
More identities can be derived through the following identities, valid for $Q^2<1$, obtained from the addition and subtraction of~\eqref{eq.gjneyfk} and~\eqref{eq.pglrqhp}:
\begin{gather}
\int_0^\pi  {\frac{{x\sin x}}{{1 - Q^4 \sin ^4 x}}\,dx}  = \frac{\pi }{{2Q\sqrt {1 - Q^2 } }}\arctan \left( {\frac{Q}{{\sqrt {1 - Q^2 } }}} \right) + \frac{\pi }{{2Q\sqrt {1 + Q^2 } }}\ln \left( {Q + \sqrt {1 + Q^2 } } \right),\\
\int_0^\pi  {\frac{{x\sin ^3 x}}{{1 - Q^4 \sin ^4 x}}\,dx}  = \frac{\pi }{{2Q^3 \sqrt {1 - Q^2 } }}\arctan \left( {\frac{Q}{{\sqrt {1 - Q^2 } }}} \right) - \frac{\pi }{{2Q^3 \sqrt {1 + Q^2 } }}\ln \left( {Q + \sqrt {1 + Q^2 } } \right).
\end{gather}
\end{remark}

\begin{remark}
Replacing $Q$ with $1/Q$ in \eqref{eq.tkzy7wr} yields
\begin{equation}
\int_0^{\pi/2} \frac{{\sin x}}{{1 + Q^2\sin^2 x}}\,dx = 
\frac{1}{Q\sqrt{1 + Q^2}}\ln \left( Q + \sqrt{1 + Q^2} \right),
\end{equation}
which upon comparison with \eqref{eq.gjneyfk} proves the following relation valid for all $Q$
\begin{equation}\label{eq.fm2dodr}
\int_0^{\pi/2} \frac{{\sin x}}{{1 + Q^2\sin^2 x}}\,dx = \frac{1}{\pi} \int_0^{\pi} \frac{{x \sin x}}{{1 + Q^2\sin^2 x}}\,dx.
\end{equation}
In fact,~\eqref{eq.fm2dodr} implies that
\begin{equation}
\int_0^{\pi/2} \frac{{\sin^{2m - 1} x}}{{\left( {1 + Q^2 \sin^2 x} \right)^m }}\,dx 
= \frac{1}{\pi} \int_0^\pi \frac{{x\sin^{2m - 1} x}}{{\left( {1 + Q^2 \sin ^2 x} \right)^m }}\,dx,\qquad m\in\mathbb{N},\, Q\in\mathbb C.
\end{equation}
\end{remark}

\section{Results associated with \eqref{eq_eleven}}

\begin{theorem}
If $r$ is a non-zero integer, then
\begin{align}
\int_0^{\pi/2} \frac{x^2}{L_r + 2\cos(2x)}\,dx = \frac{1}{\sqrt{5} F_r} \left( \frac{\pi^3}{24} + \frac{\pi}{2} \Li_2(\beta^r) \right),\quad\text{$r$ even}; \label{eq.id1_from_eleven} \\
\int_0^{\pi/2} \frac{x^2}{\sqrt{5}F_r - 2\cos(2x)}\,dx = \frac{1}{L_r} \left( \frac{\pi^3}{24} + \frac{\pi}{2} \Li_2(\beta^r) \right),\quad\text{$r$ odd}. \label{eq.id2_from_eleven}
\end{align}
In particular,
\begin{equation}
\int_0^{\pi/2} \frac{x^2}{3 + 2\cos(2x)}\,dx = \frac{1}{\sqrt{5}}\left ( \frac{3\pi^3}{40} - \frac{\pi}{2}\ln^2 \alpha \right )
\end{equation}
and
\begin{equation}
\int_0^{\pi/2} \frac{x^2}{\sqrt{5} - 2\cos(2x)}\,dx = \frac{\pi^3}{120} + \frac{\pi}{4}\ln^2 \alpha.
\end{equation}
\end{theorem}
\begin{proof}
Set $q=-\beta^r$ in~\eqref{eq_eleven} and use~\eqref{eq.s2rjcqb}. The special cases follow from the evaluations
\begin{equation}\label{eq.zpmxg81}
\Li_2(\beta^2) = \frac{\pi^2}{15}-\ln^2 \alpha \quad\mbox{and}\quad \Li_2(\beta) = -\frac{\pi^2}{15}+\frac{\ln^2 \alpha}{2}.
\end{equation}
\end{proof}

\begin{theorem}
If $r$ is a non-zero integer, then
\begin{align}
\int_0^{\pi/2} \frac{x^2}{(L_r + 2\cos(2x))^2}\,dx = \frac{L_r}{(\sqrt{5} F_r)^3} \left( \frac{\pi^3}{24} + \frac{\pi}{2} \Li_2(\beta^r) \right) - \frac{\pi}{2} \frac{\ln(1-\beta^r)}{5F_r^2},\quad\text{$r$ even}; \label{eq.id3_from_eleven} \\
\int_0^{\pi/2} \frac{x^2}{(\sqrt{5}F_r - 2\cos(2x))^2}\,dx = \frac{\sqrt{5}F_r}{L_r^3} \left( \frac{\pi^3}{24} + \frac{\pi}{2} \Li_2(\beta^r) \right) - \frac{\pi}{2}\frac{\ln(1-\beta^r)}{L_r^2},\quad\text{$r$ odd}. \label{eq.id4_from_eleven}
\end{align}
In particular,
\begin{equation}
\int_0^{\pi/2} \frac{x^2}{(3 + 2\cos(2x))^2}\,dx = 
\frac{3}{5\sqrt{5}}\left ( \frac{3\pi^3}{40} - \frac{\pi}{2}\ln^2 \alpha \right ) + \frac{\pi}{10}\ln \alpha
\end{equation}
and
\begin{equation}
\int_0^{\pi/2} \frac{x^2}{(\sqrt{5} - 2\cos(2x))^2}\,dx = 
\sqrt{5}\left ( \frac{\pi^3}{120} + \frac{\pi}{4}\ln^2 \alpha \right ) - \frac{\pi}{2}\ln \alpha.
\end{equation}
\end{theorem}
\begin{proof}
Differentiate the Fibonacci-Lucas function forms of~\eqref{eq.id1_from_eleven} and ~\eqref{eq.id2_from_eleven}, 
and take the real part, using~\eqref{eq.wyg8jr6}.
\end{proof}

\begin{theorem}
If $r\geq 2$ is an even integer, then
\begin{equation}
\int_0^{\pi/2} \frac{x^2}{L_{2r} + \sqrt{5}F_{2r} \cos(2x)}\,dx 
= \frac{\pi^3}{48} + \frac{\pi}{4} \Li_2\left (\frac{\sqrt{5} F_r}{L_r}\right ). \label{eq.id5_from_eleven}
\end{equation}
In particular,
\begin{equation}
\int_0^{\pi/2} \frac{x^2}{7 + 3\sqrt{5} \cos(2x)}\,dx = \frac{\pi^3}{48} + \frac{\pi}{4} \Li_2\left (\frac{\sqrt{5}}{3}\right ).
\end{equation}
\end{theorem}
\begin{proof}
Set $q=-\sqrt{5}F_r/L_r$ in~\eqref{eq_eleven} and keep in mind that $q < -1$ for $r$ being odd and $-1<q<0$ for $r$ being even.
\end{proof}

\begin{theorem}
If $r\geq 2$ is an even integer, then
\begin{equation}
\int_0^{\pi/2} \frac{x^2(\sqrt{5}F_{2r} + L_{2r}\cos(2x))}{(L_{2r} + \sqrt{5}F_{2r} \cos(2x))^2}\,dx 
= \frac{\pi}{2\sqrt{5}F_{2r}}\ln \left (1-\frac{\sqrt{5} F_r}{L_r}\right ). \label{eq.id6_from_eleven}
\end{equation}
In particular,
\begin{equation}
\int_0^{\pi/2} \frac{x^2(3\sqrt{5} + 7\cos(2x))}{(7 + 3\sqrt{5} \cos(2x))^2}\,dx =\frac{\pi}{6\sqrt{5}}\ln \left (\frac{2}{3\alpha^2}\right ).
\end{equation}
\end{theorem}
\begin{proof}
Differentiate the Fibonacci-Lucas function form of~\eqref{eq.id5_from_eleven} and take the real part, using~\eqref{eq.wyg8jr6}.
\end{proof}

\begin{theorem}
If $r\geq 2$ is an integer, then
\begin{equation}
\int_0^{\pi/2} \frac{x^2}{L_{r}^2 + 4 + 4L_{r} \cos(2x)}\,dx 
= \frac{1}{L_r^2-4}\left ( \frac{\pi^3}{24} + \frac{\pi}{2} \Li_2\left (\frac{2}{L_r}\right )\right ). \label{eq.id7_from_eleven}
\end{equation}
In particular,
\begin{equation}
\int_0^{\pi/2} \frac{x^2}{5 + 4 \cos(2x)}\,dx = \frac{\pi^3}{36} - \frac{\pi}{12}\ln^2 2.
\end{equation}
\end{theorem}
\begin{proof}
Set $q=-2/L_r$ in~\eqref{eq_eleven}. The special case follows from the evaluation 
\begin{equation*}
\Li_2\left (\frac{1}{2}\right ) = \frac{\pi^2}{12} - \frac{1}{2}\ln^2 2.
\end{equation*}
\end{proof}

\begin{theorem}
If $r\geq 2$ is an integer, then
\begin{equation}
\int_0^{\pi/2} \frac{x^2(L_r + 2\cos(2x))}{(L_{r}^2 + 4 + 4L_{r} \cos(2x))^2}\,dx 
= \frac{L_r}{(L_r^2-4)^2}\left ( \frac{\pi^3}{24} + \frac{\pi}{2} \Li_2\left (\frac{2}{L_r}\right )\right )
- \frac{\pi}{4} \frac{1}{L_r(L_r^2-4)}\ln\left (1-\frac{2}{L_r}\right ). \label{eq.id8_from_eleven}
\end{equation}
In particular,
\begin{equation}
\int_0^{\pi/2} \frac{x^2(2+\cos(2x))}{(5 + 4 \cos(2x))^2}\,dx = \frac{\pi^3}{54} - \frac{\pi}{18}\ln^2 2 + \frac{\pi}{24} \ln 2.
\end{equation}
\end{theorem}
\begin{proof}
Differentiate the Fibonacci-Lucas function form of~\eqref{eq.id7_from_eleven} and take the real part, using~\eqref{eq.wyg8jr6}.
\end{proof}

\section{Results associated with~\eqref{eq.n71rwsq}}

\begin{theorem}
If $r$ is a positive integer, then
\begin{gather}\label{eq.aekfmpm}
\frac{1}{{F_{2r} \sqrt 5 }}\left( {\frac{{\pi ^3 }}{3} + \pi \Li_2 (\beta ^{2r} )} \right) =  \begin{cases}
 \int_0^\pi  {\frac{{x^2 }}{{L_r^2  - 4\cos ^2 x}}\,dx},&\text{if $r$ is even};  \\
 \int_0^\pi  {\frac{{x^2 }}{{L_r^2  + 4\sin ^2 x}}\,dx},&\text{if $r$ is odd}.  \\
 \end{cases}
\end{gather}
In particular,
\begin{equation}
\int_0^\pi  {\frac{{x^2 }}{{1 + 4\sin ^2 x}}\,dx}  = \frac{{2\pi ^3 }}{{5\sqrt 5 }} - \frac{{\pi \ln ^2 \alpha }}{{\sqrt 5 }}.
\end{equation}
\end{theorem}
\begin{proof}
Set $q=\beta^{2r}$ in~\eqref{eq.n71rwsq} and use~\eqref{eq.s2rjcqb}. Note the use of the Fibonacci-Lucas fundamental identity 
$L_r^2 - 5F_r^2=(-1)^r 4$.
\end{proof}

\begin{corollary}\label{cor.m86skx9}
If $r$ is a positive integer, then
\begin{gather}
\int_0^\pi {\frac{{x^2 }}{{\left( {L_r^2 - 4\cos ^2 x} \right)^2 }}\,dx} = \frac{{\pi ^3 \sqrt 5 }}{{75}}\frac{{L_{2r} }}{{F_{2r}^3 }} + \frac{{\pi \sqrt 5 }}{{25}}\frac{{L_{2r} }}{{F_{2r}^3 }}\Li_2 (\beta ^{2r} ) - \frac{\pi }{{5F_{2r}^2 }}\ln \left( {\beta ^r F_r \sqrt 5 } \right),\text{ $r$ even},\\
\int_0^\pi {\frac{{x^2 }}{{\left( {L_r^2  + 4\sin ^2 x} \right)^2 }}\,dx} = \frac{{\pi ^3 \sqrt 5 }}{{75}}\frac{{L_{2r} }}{{F_{2r}^3 }} + \frac{{\pi \sqrt 5 }}{{25}}\frac{{L_{2r} }}{{F_{2r}^3 }}\Li_2 (\beta ^{2r} ) - \frac{\pi }{{5F_{2r}^2 }}\ln \left( {-\beta ^r L_r} \right),\text{ $r$ odd}.
\end{gather}
\end{corollary}
In particular,
\begin{equation}
\int_0^\pi  {\frac{{x^2 }}{{\left( {1 + 4\sin ^2 x} \right)^2 }}\,dx}  = \frac{{6\pi ^3 \sqrt 5 }}{{125}} - \frac{{3\pi \sqrt 5 }}{{25}}\ln ^2 \alpha  + \frac{\pi }{5}\ln \alpha.
\end{equation}
\begin{proof}
Differentiate the Fibonacci-Lucas function forms of~\eqref{eq.aekfmpm} with respect to $r$ and use~\eqref{eq.wyg8jr6}.
\end{proof}

Identity~\eqref{eq.n71rwsq} can also be written as
\begin{equation}\label{eq.wjz19s4}
\int_0^\pi  {\frac{{x^2 }}{{1 - Q\cos ^2 x}}\,dx} = \frac{1}{{\sqrt {1 - Q} }}\left( {\frac{{\pi ^3 }}{3} + \pi \Li_2 \left( {\frac{{2 - Q - 2\sqrt {1 - Q} }}{Q}} \right)} \right),Q < 1;
\end{equation}
from which we can obtain more results.

\begin{theorem}
If $r$ is a non-zero integer, then
\begin{equation}\label{eq.wbnqdef}
\int_0^\pi  {\frac{{x^2 }}{{L_r^2  - 4( - 1)^r \cos ^2 x}}\,dx}  = \frac{1}{{F_{2r} \sqrt 5 }}\left( {\frac{{\pi ^3 }}{3} + \pi \Li_2 \left( {( - 1)^r \beta ^{2r} } \right)} \right).
\end{equation}
\end{theorem}
\begin{proof}
Setting $Q=\sin^2z$ in~\eqref{eq.wjz19s4} gives
\begin{equation}\label{eq.ed6ieo7}
\int_0^\pi  {\frac{{x^2 }}{{1 - \sin ^2 z\cos ^2 x}}\,dx}  = \frac{1}{{\cos z}}\left( {\frac{{\pi ^3 }}{3} + \pi \Li_2 \left( {\frac{{(1 - \cos z)^2 }}{{\sin ^2 z}}} \right)} \right),
\end{equation}
from which~\eqref{eq.wbnqdef} follows upon use of~\eqref{eq.j5m1ag9}.
\end{proof}

\begin{corollary}
If $r$ is a non-zero integer, then
\begin{equation}
\begin{split}
\int_0^\pi  {\frac{{x^2 }}{{\left( {L_r^2  - 4( - 1)^r \cos ^2 x} \right)^2 }}\,dx}  &= \frac{{\pi ^3 \sqrt 5 }}{{75}}\frac{{L_{2r} }}{{F_{2r}^3 }} + \frac{{\pi \sqrt 5 }}{{25}}\frac{{L_{2r} }}{{F_{2r}^3 }}\Li_2 \left( {( - 1)^r \beta ^{2r} } \right)\\
&\qquad - \frac{\pi }{{5F_{2r}^2 }}\ln \left( {1 - ( - 1)^r \beta ^{2r} } \right).
\end{split}
\end{equation}
\end{corollary}
\begin{proof}
Differentiate the Fibonacci-Lucas function form of~\eqref{eq.wbnqdef}, using~\eqref{eq.wyg8jr6}.
\end{proof}

\begin{theorem}
If $r$ is a non-zero integer, then
\begin{equation}\label{eq.ef4nkhy}
\begin{split}
\int_0^\pi  {\frac{{x^2 \cos ^2 x}}{{\left( {L_r^2  - 4( - 1)^r \cos ^2 x} \right)^2 }}\,dx} & = \left( { - \frac{1}{{10}}\frac{\pi }{{F_r F_{2r} \beta ^r }} + \frac{{\pi \sqrt 5 }}{{20}}\frac{{( - 1)^r }}{{F_{2r} }}} \right)\ln \left( {1 - ( - 1)^r \beta ^{2r} } \right)\\
&\qquad + \frac{{\pi ^3 \sqrt 5 }}{{150F_{2r} F_r^2 }} + \frac{{\pi \sqrt 5 }}{{50F_{2r} F_r^2 }}\Li_2 \left( {( - 1)^r \beta ^{2r} } \right).
\end{split}
\end{equation}
\end{theorem}
\begin{proof}
Differentiate~\eqref{eq.ed6ieo7} with respect to $z$ and use~\eqref{eq.j5m1ag9}.
\end{proof}

\begin{remark}
Equivalent/similar results to~\eqref{eq.ef4nkhy} can be obtained directly by substituting $q=\beta^{2r}$ in the following identity obtained by differentiating~\eqref{eq.n71rwsq} with respect to $q$:
\begin{equation*}
\int_0^\pi \frac{{x^2\cos^2x }}{{(1 - Q\cos ^2 x)^2 }}\,dx
= - \frac{\pi }{4}\frac{{(1 + q)^4 }}{{(1 - q)^2 }}\frac{{\ln (1 - q)}}{q} + \frac{1}{2}\left( {\frac{{1 + q}}{{1 - q}}} \right)^3 \left( {\frac{{\pi ^3 }}{3} + \pi \Li_2 (q)} \right).
\end{equation*}
\end{remark}

\section{Results associated with~\eqref{eq.lh0gp48}}

\begin{lemma}\label{lem.vinbzq4}
Let $q<1$ and let
\begin{equation}\label{eq.wlciqf5}
Q=\frac{2q}{1 + q^2}.
\end{equation}
Then
\begin{equation}
\int_0^\pi  {\frac{{x^2 \,dx}}{{1 \pm Q\cos (2x)}}}  = \frac{{1 + q^2 }}{{1 - q^2 }}\left( {\frac{{\pi ^3 }}{3} + \pi \Li_2 ( \pm q)} \right)\label{eq.k4m3n7l}.
\end{equation}
\end{lemma}
\begin{theorem}
If $r$ is an integer, then
\begin{equation}\label{eq.g8ugny7}
\frac{{\pi ^3 }}{3} + \pi \Li_2 \left( { \pm \beta ^r } \right) =  \begin{cases}
 F_r \sqrt 5 \int_0^\pi  {\dfrac{{x^2 \,dx}}{{L_r  \mp 2\cos (2x)}}},&\text{if $r$ is even};  \\ 
 L_r \int_0^\pi  {\dfrac{{x^2 \,dx}}{{F_r \sqrt 5  \pm 2\cos (2x)}}},&\text{if $r$ is odd}.  \\ 
 \end{cases} 
\end{equation}
\end{theorem}
\begin{proof}
Set $q=\beta^ r$ in~\eqref{eq.k4m3n7l} and use~\eqref{eq.s2rjcqb}.
\end{proof}
In particular,
\begin{gather}
\int_0^\pi  {\frac{{x^2 \,dx}}{{\sqrt 5  + 2\cos (2x)}}}  = \frac{{4\pi ^3 }}{{15}} + \frac{\pi }{2}\ln ^2 \alpha,\\ 
\int_0^\pi  {\frac{{x^2 \,dx}}{{\sqrt 5  - 2\cos (2x)}}}  = \frac{{13\pi ^3 }}{{30}} - \pi \ln ^2 \alpha,\\ 
\int_0^\pi  {\frac{{x^2 \,dx}}{{3 - 2\cos (2x)}}}  = \frac{1}{{\sqrt 5 }}\left( {\frac{{2\pi ^3 }}{5} - \pi \ln ^2 \alpha } \right);
\end{gather}
where we used~\eqref{eq.zpmxg81} and also
\begin{equation}\label{eq.noimci0}
\Li_2 ( - \beta ) = \frac{{\pi ^2 }}{{10}} - \ln ^2 \alpha ,\quad \Li_2 (\beta ) =  - \frac{{\pi ^2 }}{{15}} + \frac{1}{2}\ln ^2 \alpha.
\end{equation}
\begin{theorem}
If $r$ is a positive even integer, then
\begin{equation}\label{eq.kniijoy}
\int_0^\pi  {\frac{{x^2 \,dx}}{{\left( {L_r  \mp 2\cos (2x)} \right)^2 }}}  =  - \frac{\pi }{{5F_r^2 }}\ln \left( {1 \mp \beta ^r } \right) + \frac{{L_r }}{{5F_r^3 \sqrt 5 }}\left( {\frac{{\pi ^3 }}{3} + \pi \Li_2 \left( { \pm \beta ^r } \right)} \right),
\end{equation}
while if $r$ is a positive odd integer, then
\begin{equation}\label{eq.m64m49c}
\int_0^\pi  {\frac{{x^2 \,dx}}{{\left( {F_r \sqrt 5  \pm 2\cos (2x)} \right)^2 }}}  =  - \frac{\pi }{{L_r^2 }}\ln \left( {1 \mp \beta ^r } \right) + \frac{{F_r \sqrt 5 }}{{L_r^3 }}\left( {\frac{{\pi ^3 }}{3} + \pi \Li_2 \left( { \pm \beta ^r } \right)} \right).
\end{equation}
\end{theorem}
\begin{proof}
Differentiate the Fibonacci-Lucas function forms of~\eqref{eq.g8ugny7}.
\end{proof}
In particular,
\begin{gather}
\int_0^\pi  {\frac{{x^2 \,dx}}{{\left( {3 - 2\cos (2x)} \right)^2 }}}  = \frac{\pi }{5}\ln \alpha  + \frac{{6\pi ^3 }}{{25\sqrt 5 }} - \frac{{3\pi }}{{5\sqrt 5 }}\ln ^2 \alpha,\\
\int_0^\pi  {\frac{{x^2 \,dx}}{{\left( {\sqrt 5  + 2\cos (2x)} \right)^2 }}}  =  - \pi \ln \alpha  + \frac{{4\pi ^3 }}{{3\sqrt 5 }} + \frac{{\pi \sqrt 5 }}{2}\ln ^2 \alpha,\\
\int_0^\pi  {\frac{{x^2 \,dx}}{{\left( {\sqrt 5  - 2\cos (2x)} \right)^2 }}}  = 2\pi \ln \alpha  + \frac{{13\pi ^3 }}{{6\sqrt 5 }} - \pi \sqrt 5 \ln ^2 \alpha.
\end{gather}

\begin{corollary}
If $r$ is an even integer, then
\begin{gather}
\int_0^\pi  {\frac{{\left( {L_r^2 + 4\cos^2 (2x)} \right)x^2 }}{{\left( {L_r^2 - 4\cos ^2 (2x)} \right)^2 }}\,dx} = - \frac{\pi }{{10F_r^2 }}\ln \left( {\beta ^r F_r \sqrt 5 } \right) + \frac{{L_r }}{{20F_r^3 \sqrt 5 }}\left( {\frac{{4\pi ^3 }}{3} + \pi \Li_2 \left( {\beta ^{2r} } \right)} \right),\label{eq.frltbqe}\\
\int_0^\pi  {\frac{{x^2 \cos (2x)}}{{\left( {L_r^2 - 4\cos ^2 (2x)} \right)^2 }}\,dx} = \frac{\pi }{{40F_r^2 L_r }}\ln \left( {\frac{{1 + \beta ^r }}{{1 - \beta ^r }}} \right) + \frac{\pi }{{40F_r^3 \sqrt 5 }}\left( {\Li_2 \left( {\beta ^r } \right) - \Li_2 \left( { - \beta ^r } \right)} \right)\label{eq.s7jkwms}
\end{gather}
while if $r$ is an odd integer, then
\begin{gather}
\int_0^\pi  {\frac{{\left( {5F_r^2 + 4\cos^2 (2x)} \right)x^2 }}{{\left( {5F_r^2 - 4\cos ^2 (2x)} \right)^2 }}\,dx} = - \frac{\pi }{{2L_r^2 }}\ln \left( { - \beta ^r L_r } \right) + \frac{{F_r \sqrt 5 }}{{4L_r^3 }}\left( {\frac{{4\pi ^3 }}{3} + \pi \Li_2 \left( {\beta ^{2r} } \right)} \right)\label{eq.vxzm3gu},\\
\int_0^\pi {\frac{{x^2 \cos (2x)\,dx}}{{\left( {5F_r^2 - 4\cos ^2 (2x)} \right)^2 }}} = \frac{\pi }{{8L_r^2 F_r \sqrt 5 }}\ln \left( {\frac{{1 - \beta ^r }}{{1 + \beta ^r }}} \right) - \frac{\pi }{{8L_r^3 }}\left( {\Li_2 \left( {\beta ^r } \right) - \Li_2 \left( { - \beta ^r } \right)} \right)\label{eq.nt2noan}.
\end{gather}
\end{corollary}
\begin{proof}
Identities~\eqref{eq.frltbqe} and~\eqref{eq.s7jkwms} are obtained from the respective addition and subtraction of the two identities contained in~\eqref{eq.kniijoy} while~\eqref{eq.vxzm3gu} and~\eqref{eq.nt2noan} follow from~\eqref{eq.m64m49c}.
\end{proof}
In particular,
\begin{gather}
\int_0^\pi  {\frac{{\left( {5 + 4\cos^2 (2x)} \right)x^2 }}{{\left( {5 - 4\cos ^2 (2x)} \right)^2 }}\,dx} = \frac{\pi }{2}\ln \alpha + \frac{{7\pi ^3 }}{{4\sqrt 5 }} - \frac{{\pi \sqrt 5 }}{4}\ln^2\alpha,\\
\int_0^\pi  {\frac{{x^2 \cos (2x)}}{{\left( {5 - 4\cos ^2 (2x)} \right)^2 }}\,dx} = \frac{{3\pi }}{{8\sqrt 5 }}\ln \alpha + \frac{{\pi ^3 }}{{48}} - \frac{{3\pi }}{{16}}\ln ^2 \alpha.
\end{gather}

\begin{lemma}
If $0<q<1$, then
\begin{gather}
\int_0^\pi  {\frac{{x^2 \,dx}}{{1 - Q^2 \cos^2 (2x)}}} = \frac{{1 + q^2 }}{{1 - q^2 }}\left( {\frac{{\pi ^3 }}{3} + \frac{\pi }{4}\Li_2 (q^2 )} \right),\label{eq.aria0wt}\\
\int_0^\pi  {\frac{{x^2 \cos (2x)\,dx}}{{1 - Q^2 \cos^2 (2x)}}} = \frac{\pi }{{2Q}}\frac{{1 + q^2 }}{{1 - q^2 }}\left( {\Li_2 (q) - \Li_2 ( - q)} \right)\label{eq.ru2ayab},
\end{gather}
where $Q$ is as given in~\eqref{eq.wlciqf5}.
\end{lemma}
\begin{proof}
Immediate consequence of the identities in Lemma~\ref{lem.vinbzq4}. We also used
\begin{equation*}
\Li_2 (y) + \Li_2 ( - y) = \frac{1}{2}\Li_2 (y^2 ).
\end{equation*}
\end{proof}

\begin{theorem}
If $r$ is a positive integer, then
\begin{equation}\label{eq.jivtzpl}
\frac{1}{{F_{2r} \sqrt 5 }}\left( {\frac{{\pi ^3 }}{3} + \frac{\pi }{4}\Li_2 \left( {\beta ^{2r} } \right)} \right) =  \begin{cases}
 \int_0^\pi  {\dfrac{{x^2 \,dx}}{{L_r^2  - 4\cos^2 (2x)}}},&\text{if $r$ is even};  \\ 
 \int_0^\pi  {\dfrac{{x^2 \,dx}}{{L_r^2  + 4\sin^2 (2x)}}},&\text{if $r$ is odd};  \\ 
 \end{cases} 
\end{equation}
and
\begin{equation}\label{eq.isekz49}
\frac{\pi }{4}\left( {\Li_2 \left( {\beta ^r } \right) - \Li_2 \left( { - \beta ^r } \right)} \right) =  \begin{cases}
 F_r \sqrt 5 \int_0^\pi  {\dfrac{{x^2 \cos (2x)\,dx}}{{5F_r^2 + 4\sin^2 (2x)}}},&\text{if $r$ is even};  \\ 
  - L_r \int_0^\pi  {\dfrac{{x^2 \cos (2x)\,dx}}{{5F_r^2 - 4\cos^2 (2x)}}},&\text{if $r$ is odd}.  \\ 
 \end{cases} 
\end{equation}
\end{theorem}
\begin{proof}
Set $q=\beta^r$ in~\eqref{eq.aria0wt} and~\eqref{eq.ru2ayab}.
\end{proof}
In particular,
\begin{gather}
\int_0^\pi  {\frac{{x^2 \,dx}}{{1 + 4\sin^2 (2x)}}} = \frac{1}{{\sqrt 5 }}\left( {\frac{{7\pi ^3 }}{{20}} - \frac{\pi}{4}\ln ^2 \alpha } \right),\\
\int_0^\pi  {\frac{{x^2 \cos (2x)\,dx}}{{5 - 4\cos^2 (2x)}}} = \frac{{\pi ^3 }}{{24}} - \frac{{3\pi }}{8}\ln ^2 \alpha.
\end{gather}

\begin{corollary}
If $r$ is a positive even integer, then
\begin{gather}
\int_0^\pi  {\frac{{x^2 \cos ^2 x}}{{L_r^2 - 4\cos^2 (2x)}}\,dx}  = \frac{1}{{F_{2r} \sqrt 5 }}\left( {\frac{{\pi ^3 }}{6} + \frac{\pi }{8}\Li_2 \left( {\beta ^{2r} } \right)} \right) + \frac{\pi }{{8F_r \sqrt 5 }}\left( {\Li_2 \left( {\beta ^r } \right) - \Li_2 \left( { - \beta ^r } \right)} \right),\\
\int_0^\pi  {\frac{{x^2 \sin^2 x}}{{L_r^2 - 4\cos ^2 (2x)}}\,dx} = \frac{1}{{F_{2r} \sqrt 5 }}\left( {\frac{{\pi ^3 }}{6} + \frac{\pi }{8}\Li_2 \left( {\beta ^{2r} } \right)} \right) - \frac{\pi }{{8F_r \sqrt 5 }}\left( {\Li_2 \left( {\beta ^r } \right) - \Li_2 \left( { - \beta ^r } \right)} \right),
\end{gather}
while if $r$ is a positive odd number, then
\begin{gather}
\int_0^\pi  {\frac{{x^2 \sin^2 x}}{{L_r^2 + 4\sin ^2 (2x)}}\,dx} = \frac{1}{{F_{2r} \sqrt 5 }}\left( {\frac{{\pi ^3 }}{6} + \frac{\pi }{8}\Li_2 \left( {\beta ^{2r} } \right)} \right) + \frac{\pi }{{8L_r }}\left( {\Li_2 \left( {\beta ^r } \right) - \Li_2 \left( { - \beta ^r } \right)} \right),\\
\int_0^\pi  {\frac{{x^2 \cos^2 x}}{{L_r^2 + 4\sin ^2 (2x)}}\,dx} = \frac{1}{{F_{2r} \sqrt 5 }}\left( {\frac{{\pi ^3 }}{6} + \frac{\pi }{8}\Li_2 \left( {\beta ^{2r} } \right)} \right) - \frac{\pi }{{8L_r }}\left( {\Li_2 \left( {\beta ^r } \right) - \Li_2 \left( { - \beta ^r } \right)} \right).
\end{gather}
\end{corollary}
\begin{proof}
Addition and subtraction of corresponding identities in~\eqref{eq.jivtzpl} and~\eqref{eq.isekz49}.
\end{proof}
In particular,
\begin{gather}
\int_0^\pi  {\frac{{x^2 \sin^2 x}}{{1 + 4\sin^2 (2x)}}\,dx} = \left( { - \frac{{\sqrt 5 }}{{40}} + \frac{3}{{16}}} \right)\pi \ln ^2 \alpha  + \left( {\frac{{7\sqrt 5 }}{{200}} - \frac{1}{{48}}} \right)\pi ^3,\\
\int_0^\pi  {\frac{{x^2 \cos^2 x}}{{1 + 4\sin^2 (2x)}}\,dx} = - \left( {\frac{{\sqrt 5 }}{{40}} + \frac{3}{{16}}} \right)\pi \ln ^2 \alpha  + \left( {\frac{{7\sqrt 5 }}{{200}} + \frac{1}{{48}}} \right)\pi ^3.
\end{gather}
\begin{theorem}
If $r$ is a positive integer, then
\begin{equation}\label{eq.pxi3hd5}
\begin{split}
&\frac{{L_{2r} }}{{5F_{2r}^3 \sqrt 5 }}\left( {\frac{{\pi ^3 }}{3} + \frac{\pi }{4}\Li_2 \left( { \beta ^{2r} } \right)} \right) - \frac{\pi }{{20F_{2r}^2 }}\ln \left( {1 - \beta ^{2r} } \right)\\
&\qquad =  \begin{cases}
 \int_0^\pi  {\dfrac{{x^2 \,dx}}{{\left( {L_r^2  - 4\cos^2 (2x)} \right)^2 }}},&\text{if $r$ is even};  \\ 
 \int_0^\pi  {\dfrac{{x^2 \,dx}}{{\left( {L_r^2  + 4\sin^2 (2x)} \right)^2 }}},&\text{if $r$ is odd} ; \\ 
 \end{cases} 
\end{split}
\end{equation}
and
\begin{equation}\label{eq.ezyw57r}
\begin{split}
&\frac{\pi }{{8F_{2r} \sqrt 5 }}\ln \left( {\frac{{1 + \beta ^r }}{{1 - \beta ^r }}} \right)\\
&\qquad =  \begin{cases}
  - \dfrac{\pi }{{40F_r^2 }}\left( {\Li_2 \left( {\beta ^r } \right) - \Li_2 \left( { - \beta ^r } \right)} \right) + F_r \sqrt 5 \int_0^\pi  {\dfrac{{x^2 \cos (2x)\,dx}}{{\left( {5F_r^2  + 4\sin ^2 (2x)} \right)^2 }}},&\text{$r$ even};  \\ 
  - \dfrac{\pi }{{8L_r^2 }}\left( {\Li_2 \left( {\beta ^r } \right) - \Li_2 \left( { - \beta ^r } \right)} \right) - L_r \int_0^\pi  {\dfrac{{x^2 \cos (2x)\,dx}}{{\left( {5F_r^2  - 4\cos ^2 (2x)} \right)^2 }}},&\text{$r$ odd}.  \\ 
 \end{cases} 
\end{split}
\end{equation}
\end{theorem}
\begin{proof}
Differentiate the Fibonacci-Lucas function form of~\eqref{eq.jivtzpl} to obtain~\eqref{eq.pxi3hd5}; differentiate the Fibonacci-Lucas function form of~\eqref{eq.isekz49} to obtain~\eqref{eq.ezyw57r}.
\end{proof}
In particular,
\begin{gather}
\int_0^\pi  {\frac{{x^2 \,dx}}{{\left( {1 + 4\sin^2 (2x)} \right)^2 }}} = \frac{{21}}{{100}}\frac{{\pi ^3 }}{{\sqrt 5 }} - \frac{{3\pi }}{{20\sqrt 5 }}\ln ^2 \alpha  + \frac{\pi }{{20}}\ln \alpha, \\
\int_0^\pi  {\frac{{x^2 \cos (2x)\,dx}}{{\left( {5 - 4\cos^2 (2x)} \right)^2 }}} = \frac{{3\pi }}{{8\sqrt 5 }}\ln \alpha  + \frac{{\pi ^3 }}{{48}} - \frac{{3\pi }}{{16}}\ln ^2 \alpha. 
\end{gather}

\begin{lemma}\label{lem.e9mcx5b}
Let $0<q<1$. Let
\begin{equation*}
R=\frac{2q}{1 - q^2}.
\end{equation*}
Then
\begin{equation}\label{eq.xitqgr6}
\int_0^\pi  {\frac{{x^2 \,dx}}{{1 + R^2 \cos^2 (2x)}}} = \frac{{1 - q^2 }}{{1 + q^2 }}\left( {\frac{{\pi ^3 }}{3} + \frac{\pi }{4}\Li_2 \left( { - q^2 } \right)} \right)
\end{equation}
and
\begin{equation}\label{eq.d64v4ze}
\begin{split}
&\int_0^\pi  {\frac{{x^2\cos(2x) \,dx}}{{1 + R^2 \cos^2 (2x)}}} \\
&\qquad = \frac{\pi }{R}\frac{{1 - q^2 }}{{1 + q^2 }}\left( {\arctan q\ln q + \frac{1}{2}\Cl_2 \left( {2\arctan q} \right) + \frac{1}{2}\Cl_2 \left( {\pi - 2\arctan q} \right)} \right).
\end{split}
\end{equation}
\end{lemma}
\begin{proof}
Write $iq$ for $q$ in~\eqref{eq.lh0gp48} and take real and imaginary parts to obtain
\begin{gather*}
\int_0^\pi  {\frac{{x^2 \,dx}}{{1 + R^2 \cos^2 (2x)}}} = \frac{{1 - q^2 }}{{1 + q^2 }}\left( {\frac{{\pi ^3 }}{3} + \pi \Re \Li_2 \left( {iq} \right)} \right),\\
\int_0^\pi  {\frac{{x^2 \cos (2x)\,dx}}{{1 + R^2 \cos^2 (2x)}}} = \frac{\pi }{R}\frac{{1 - q^2 }}{{1 + q^2 }}\Im \Li_2 (iq),
\end{gather*}
from which~\eqref{eq.xitqgr6} and~\eqref{eq.d64v4ze} follow upon using Lemma~\ref{lem.oszn90v}.
\end{proof}
\begin{theorem}
If $r$ is a positive integer, then
\begin{equation}\label{eq.eunbs0s}
\frac{1}{{F_{2r} \sqrt 5 }}\left( {\frac{{\pi ^3 }}{3} + \frac{\pi }{4}\Li_2 \left( { - \beta ^{2r} } \right)} \right) =  \begin{cases}
 \int_0^\pi  {\dfrac{{x^2 \,dx}}{{5F_r^2 + 4\cos^2 (2x)}}},&\text{if $r$ is even};  \\ 
 \int_0^\pi  {\dfrac{{x^2 \,dx}}{{L_r^2 + 4\cos^2 (2x)}}},&\text{if $r$ is odd}.  \\ 
 \end{cases} 
\end{equation}

\end{theorem}
\begin{proof}
Set $q=\beta ^r$ in~\eqref{eq.xitqgr6} and use~\eqref{eq.s2rjcqb}.
\end{proof}
\begin{theorem}
If $r$ is a positive integer, then
\begin{equation}
\begin{split}
&\frac{{L_{2r} }}{{5F_{2r}^3 \sqrt 5 }}\left( {\frac{{\pi ^3 }}{3} + \frac{\pi }{4}\Li_2 \left( { - \beta ^{2r} } \right)} \right) - \frac{\pi }{{20F_{2r}^2 }}\ln \left( {1 + \beta ^{2r} } \right)\\
&\qquad =  \begin{cases}
 \int_0^\pi {\dfrac{{x^2 \,dx}}{{\left( {5F_r^2 + 4\cos^2 (2x)} \right)^2 }}},&\text{if $r$ is even};  \\ 
 \int_0^\pi {\dfrac{{x^2 \,dx}}{{\left( {L_r^2 + 4\cos^2 (2x)} \right)^2 }}},&\text{if $r$ is odd} . \\ 
 \end{cases} 
\end{split}
\end{equation}
\end{theorem}
\begin{proof}
Differentiate the Fibonacci-Lucas function form of~\eqref{eq.eunbs0s}.
\end{proof}

\begin{theorem}
If $r$ is a positive even integer, then
\begin{equation}
\begin{split}
\int_0^\pi  {\frac{{x^2 \cos (2x)\,dx}}{{5F_r^2 + 4\cos^2 (2x)}}}  &= \frac{\pi }{{4L_r }}\left( {\Cl_2 \left( {\arctan \left( {\frac{2}{{F_r \sqrt 5 }}} \right)} \right) + \Cl_2 \left( {\pi - \arctan \left( {\frac{2}{{F_r \sqrt 5 }}} \right)} \right)} \right)\\
&\qquad - \frac{{\pi r}}{{4L_r }}\arctan \left( {\frac{2}{{F_r \sqrt 5 }}} \right)\ln \alpha .
\end{split}
\end{equation}
while, if $r$ is a positive odd integer, then
\begin{equation}
\begin{split}
\int_0^\pi  {\frac{{x^2 \cos (2x)\,dx}}{{L_r^2 + 4\cos^2 (2x)}}}  &= \frac{\pi }{{4F_r\sqrt 5 }}\left( {\Cl_2 \left( {\arctan \left( {\frac{2}{{L_r }}} \right)} \right) + \Cl_2 \left( {\pi - \arctan \left( {\frac{2}{{L_r}}} \right)} \right)} \right)\\
&\qquad - \frac{{\pi r}}{{4F_r\sqrt 5 }}\arctan \left( {\frac{2}{{L_r}}} \right)\ln \alpha .
\end{split}
\end{equation}
\end{theorem}
\begin{proof}
Set $q=\pm\beta^r$ in~\eqref{eq.d64v4ze}; use~\eqref{eq.s2rjcqb} and Lemma~\ref{lem.z4tjirr}.
\end{proof}

\section{Results associated with~\eqref{eq.b1e7nal}}

\begin{theorem}\label{thm.b52bbpx}
If $r$ is a positive integer, then
\begin{gather}
\int_0^\pi  {\frac{{x^2 \cos x}}{{L_r - 2\cos (2x)}}\,dx} = \frac{{\pi \sqrt {\beta ^r } }}{{1 - \beta ^r }}\left( {\Li_2 \left( { - \sqrt {\beta ^r } } \right) - \Li_2 \left( {\sqrt {\beta ^r } } \right)} \right),\quad\text{if $r$ is even},\label{eq.ri9nkko}\\
\int_0^\pi  {\frac{{x^2 \cos x}}{{F_r \sqrt 5 - 2\cos (2x)}}\,dx} = \frac{{\pi \sqrt {-\beta ^r } }}{{1 + \beta ^r }}\left( {\Li_2 \left( { - \sqrt {-\beta ^r } } \right) - \Li_2 \left( {\sqrt {-\beta ^r } } \right)} \right),\quad\text{if $r$ is odd}\label{eq.ujhmyef}.
\end{gather}
\end{theorem}
In particular,
\begin{equation}
\int_0^\pi  {\frac{{x^2 \cos x}}{{3 - 2\cos (2x)}}\,dx} = - \frac{{\pi ^3 }}{6} + \frac{{3\pi }}{2}\ln ^2 \alpha.
\end{equation}
\begin{proof}
Set $q=\beta^r$ in~\eqref{eq.b1e7nal} to obtain~\eqref{eq.ri9nkko} and $q=-\beta^r$ to obtain~\eqref{eq.ujhmyef}. Use~\eqref{eq.s2rjcqb}.
\end{proof}

\begin{theorem}
If $r$ is an even positive integer, then
\begin{equation}
\begin{split}
&\int_0^\pi  {\frac{{x^2 \cos x}}{{\left( {L_r  - 2\cos (2x)} \right)^2 }}\,dx}\\
&\qquad  =  - \frac{\pi }{{F_r \sqrt 5 }}\,\frac{{\sqrt {\beta ^r } }}{{1 - \beta ^r }}\left( {\frac{1}{2} + \frac{{\beta ^r }}{{1 - \beta ^r }}} \right)\left( {\Li_2 \left( {\sqrt {\beta ^r } } \right) - \Li_2 \left( { - \sqrt {\beta ^r } } \right)} \right)\\
&\qquad\qquad - \frac{\pi }{{2F_r \sqrt 5 }}\,\frac{{\sqrt {\beta ^r } }}{{1 - \beta ^r }}\ln \left( {\frac{{1 + \sqrt {\beta ^r } }}{{1 - \sqrt {\beta ^r } }}} \right),
\end{split}
\end{equation}
while if $r$ is an odd positive integer, then
\begin{equation}
\begin{split}
&\int_0^\pi  {\frac{{x^2 \cos x}}{{\left( {F_r \sqrt 5 - 2\cos (2x)} \right)^2 }}\,dx}\\
  &\qquad=  - \frac{\pi }{{L_r }}\,\frac{{\sqrt { - \beta ^r } }}{{1 + \beta ^r }}\left( {\frac{1}{2} - \frac{{\beta ^r }}{{1 + \beta ^r }}} \right)\left( {\Li_2 \left( { \sqrt { - \beta ^r } } \right) - \Li_2 \left( { - \sqrt { - \beta ^r } } \right)} \right)\\
&\qquad\qquad - \frac{\pi }{{2L_r }}\,\frac{{\sqrt { - \beta ^r } }}{{1 + \beta ^r }}\ln \left( {\frac{{1 + \sqrt { - \beta ^r } }}{{1 - \sqrt { - \beta ^r } }}} \right).
\end{split}
\end{equation}
\end{theorem}
In particular,
\begin{equation}
\int_0^\pi  {\frac{{x^2 \cos x}}{{\left( {3 - 2\cos (2x)} \right)^2 }}} = - \frac{\pi }{4}\left( {\frac{{\pi ^2 }}{3} - 3\ln ^2 \alpha } \right) - \frac{{3\pi }}{{2\sqrt 5 }}\ln \alpha,
\end{equation}
where we used~\eqref{eq.noimci0}.
\begin{proof}
Differentiate the Fibonacci-Lucas function forms of the results in Theorem~\ref{thm.b52bbpx} and take the real part.
\end{proof}

The next results involve the Clausen function.
\begin{theorem}
If $r$ is a positive even integer, then
\begin{equation}\label{eq.pi7i3yl}
\begin{split}
&\int_0^\pi  {\frac{{x^2 \cos x}}{{L_r + 2\cos (2x)}}\,dx}\\
 &\qquad =  - \frac{{\pi \sqrt {\beta ^r } }}{{1 + \beta ^r }}\left( {2\arctan \left( {\sqrt {\beta ^r } } \right)\ln \left( {\sqrt {\beta ^r } } \right)} \right)\\
&\qquad\qquad - \frac{{\pi \sqrt {\beta ^r } }}{{1 + \beta ^r }}\left( {\Cl_2 \left( {2\arctan \left( {\sqrt {\beta ^r } } \right)} \right) + \Cl_2 \left( {\pi  - 2\arctan \left( {\sqrt {\beta ^r } } \right)} \right)} \right)
\end{split}
\end{equation}
while if $r$ is a positive odd integer, then
\begin{equation}\label{eq.a40fggg}
\begin{split}
&\int_0^\pi  {\frac{{x^2 \cos x}}{{F_r \sqrt 5 + 2\cos (2x)}}\,dx}\\
 &\qquad =  - \frac{{\pi \sqrt {-\beta ^r } }}{{1 -\beta ^r }}\left( {2\arctan \left( {\sqrt {-\beta ^r } } \right)\ln \left( {\sqrt {-\beta ^r } } \right)} \right)\\
&\qquad\qquad - \frac{{\pi \sqrt {-\beta ^r } }}{{1 -\beta ^r }}\left( {\Cl_2 \left( {2\arctan \left( {\sqrt {-\beta ^r } } \right)} \right) + \Cl_2 \left( {\pi  - 2\arctan \left( {\sqrt {-\beta ^r } } \right)} \right)} \right)
\end{split}
\end{equation}
\end{theorem}
In particular,
\begin{equation}
\begin{split}
\int_0^\pi  {\frac{{x^2 \cos x}}{{3 + 2\cos (2x)}}\,dx}  &= \frac{\pi }{{\sqrt 5 }}\,\arctan 2\ln \alpha \\
&\qquad - \frac{\pi }{{\sqrt 5 }}\,\left( {\Cl_2 (\arctan 2) + \Cl_2 (\pi  - \arctan 2)} \right),
\end{split}
\end{equation}
where we used
\begin{equation}
\arctan(-\beta)=\frac12\,\arctan 2.
\end{equation}
\begin{proof}
Use of~$q=-\beta^r$ in~\eqref{eq.b1e7nal} produces
\begin{equation}\label{eq.r2cop4c}
\int_0^\pi  {\frac{{x^2 \cos x}}{{L_r + 2\cos (2x)}}\,dx} = \frac{{\pi i\sqrt { \beta ^r } }}{{1 + \beta ^r }}\left( {\Li_2 \left( {i\sqrt { \beta ^r } } \right) - \Li_2 \left( { - i\sqrt {\beta ^r } } \right)} \right),\quad\text{$r$ even};
\end{equation}
and hence~\eqref{eq.pi7i3yl} in view of~\eqref{eq.j0jplsn}.
Setting $q=\beta^r$ in~\eqref{eq.b1e7nal} gives
\begin{equation}\label{eq.gycwhuh}
\int_0^\pi  {\frac{{x^2 \cos x}}{{F_r \sqrt 5 + 2\cos (2x)}}\,dx} = \frac{{\pi \sqrt {\beta ^r } }}{{1 - \beta ^r }}\left( {\Li_2 \left( {\sqrt {\beta ^r } } \right) - \Li_2 \left( { - \sqrt {\beta ^r } } \right)} \right),\quad\text{$r$ odd},
\end{equation}
which, since $\sqrt{\beta ^r}=i\sqrt{-\beta ^r}$ for odd $r$, can also be written as
\begin{equation}
\int_0^\pi  {\frac{{x^2 \cos x}}{{F_r \sqrt 5 + 2\cos (2x)}}\,dx} = \frac{{\pi i\sqrt { - \beta ^r } }}{{1 - \beta ^r }}\left( {\Li_2 \left( {i\sqrt { - \beta ^r } } \right) - \Li_2 \left( { - i\sqrt { - \beta ^r } } \right)} \right);
\end{equation}
from which~\eqref{eq.a40fggg} follows on account of~\eqref{eq.j0jplsn}.
\end{proof}

\begin{theorem}
If $r$ is an even positive integer, then
\begin{equation}
\begin{split}
&\int_0^\pi  {\frac{{x^2 \cos x}}{{\left( {L_r + 2\cos (2x)} \right)^2 }}\,dx} \\
&\qquad =  - \frac{{2\pi }}{{F_r \sqrt 5 }}\frac{{\sqrt {\beta ^r } }}{{1 + \beta ^r }}\left( {\frac{1}{2} - \frac{{\beta ^r }}{{1 + \beta ^r }}} \right)\arctan \left( {\sqrt {\beta ^r } } \right)\ln \left( {\sqrt {\beta ^r } } \right)\\
&\qquad\quad - \frac{\pi }{{F_r \sqrt 5 }}\frac{{\sqrt {\beta ^r } }}{{1 + \beta ^r }}\left( {\frac{1}{2} - \frac{{\beta ^r }}{{1 + \beta ^r }}} \right)\left( {\Cl_2 \left( {2\arctan \left( {\sqrt {\beta ^r } } \right)} \right) - \Cl_2 \left( {\pi  - 2\arctan \left( {\sqrt {\beta ^r } } \right)} \right)} \right)\\
&\qquad\qquad - \frac{\pi }{{2F_r \sqrt 5 }}\frac{{\sqrt {\beta ^r } }}{{1 + \beta ^r }}\arctan \left( {\frac{{2\sqrt {\beta ^r } }}{{1 - \beta ^r }}} \right),
\end{split}
\end{equation}
while if $r$ is an odd positive integer, then
\begin{equation}
\begin{split}
&\int_0^\pi  {\frac{{x^2 \cos x}}{{\left( {F_r\sqrt 5 + 2\cos (2x)} \right)^2 }}\,dx} \\
&\qquad =  - \frac{{2\pi }}{L_r}\frac{{\sqrt {-\beta ^r } }}{{1 -\beta ^r }}\left( {\frac{1}{2} + \frac{{\beta ^r }}{{1 -\beta ^r }}} \right)\arctan \left( {\sqrt {-\beta ^r } } \right)\ln \left( {\sqrt {-\beta ^r } } \right)\\
&\qquad\quad - \frac{\pi }{L_r}\frac{{\sqrt {-\beta ^r } }}{{1 -\beta ^r }}\left( {\frac{1}{2} + \frac{{\beta ^r }}{{1 -\beta ^r }}} \right)\left( {\Cl_2 \left( {2\arctan \left( {\sqrt {-\beta ^r } } \right)} \right) - \Cl_2 \left( {\pi  - 2\arctan \left( {\sqrt {-\beta ^r } } \right)} \right)} \right)\\
&\qquad\qquad - \frac{\pi }{{2L_r }}\frac{{\sqrt {-\beta ^r } }}{{1 -\beta ^r }}\arctan \left( {\frac{{2\sqrt {-\beta ^r } }}{{1 + \beta ^r }}} \right).
\end{split}
\end{equation}
\end{theorem}
\begin{proof}
Differentiate the Fibonacci-Lucas function forms of~\eqref{eq.r2cop4c} and~\eqref{eq.gycwhuh} and take the real part in each case.
\end{proof}

\begin{lemma}
If $r$ is a positive integer, then
\begin{equation}\label{eq.qvb6jur}
\begin{split}
&\frac{1}{{2F_r \sqrt 5 }}\int_0^\pi  {\frac{{x^2 \cos x}}{{F_r \sqrt 5  - 2\cos (2x)}}\,dx} + \frac{1}{{2F_r \sqrt 5 }}\int_0^\pi {\frac{{x^2 \cos x}}{{F_r \sqrt 5 + 2\cos (2x)}}\,dx} \\
&\qquad = \int_0^\pi  {\frac{{x^2 \cos x}}{{5F_r^2 - 4\cos^2 (2x)}}\,dx} 
\end{split}
\end{equation}
and
\begin{equation}\label{eq.a40qd9a}
\begin{split}
&\frac{1}{2}\int_0^\pi  {\frac{{x^2 \cos x}}{{F_r \sqrt 5 - 2\cos (2x)}}\,dx} + \frac{1}{2}\int_0^\pi  {\frac{{x^2 \cos x}}{{F_r \sqrt 5 + 2\cos (2x)}}\,dx} \\
&\qquad= \int_0^\pi  {\frac{{x^2 \cos x}}{{5F_r^2 - 4\cos^2 (2x)}}\,dx} + \int_0^\pi {\frac{{x^2 \cos (3x)}}{{5F_r^2 - 4\cos^2 (2x)}}\,dx} .
\end{split}
\end{equation}
\end{lemma}
\begin{proof}
Identity~\eqref{eq.qvb6jur} is obvious while~\eqref{eq.a40qd9a} becomes clear once the elementary trigonometric identity 
$2\cos x\cos{(2x)}=\cos{(3x)} + \cos x$ is employed.
\end{proof}

\begin{theorem}
If $r$ is a positive odd integer, then
\begin{equation}\label{eq.y4dqfp7}
\begin{split}
&\int_0^\pi  {\frac{{x^2 \cos x}}{{5F_r^2 - 4\cos^2 (2x)}}\,dx} \\
&\qquad=  - \frac{\pi }{{2F_r \sqrt 5 }}\frac{{\sqrt { - \beta ^r } }}{{1 + \beta ^r }}\left( {\Li_2 \left( {\sqrt { - \beta ^r } } \right) - \Li_2 \left( { - \sqrt { - \beta ^r } } \right)} \right)\\
&\qquad\quad - \frac{\pi }{{2F_r \sqrt 5 }}\frac{{\sqrt { - \beta ^r } }}{{1 - \beta ^r }}\left( {\Cl_2 \left( {2\arctan \left( {\sqrt { - \beta ^r } } \right)} \right) + \Cl_2 \left( {\pi  - 2\arctan \left( {\sqrt { - \beta ^r } } \right)} \right)} \right)\\
&\qquad\qquad - \frac{\pi }{{F_r \sqrt 5 }}\frac{{\sqrt { - \beta ^r } }}{{1 - \beta ^r }}\arctan \left( {\sqrt { - \beta ^r } } \right)\ln \left( {\sqrt { - \beta ^r } } \right)
\end{split}
\end{equation}
and
\begin{equation}\label{eq.gg6a1vx}
\begin{split}
&\int_0^\pi  {\frac{{x^2 \cos (3x)}}{{5F_r^2 - 4\cos^2 (2x)}}\,dx} \\
&\qquad = \left( {\frac{1}{{F_r \sqrt 5 }} - 1} \right)\frac{\pi }{2}\frac{{\sqrt { - \beta ^r } }}{{1 + \beta ^r }}\left( {\Li_2 \left( {\sqrt { - \beta ^r } } \right) - \Li_2 \left( { - \sqrt { - \beta ^r } } \right)} \right)\\
&\qquad\quad+ \left( {\frac{1}{{F_r \sqrt 5 }} + 1} \right)\frac{\pi }{2}\frac{{\sqrt { - \beta ^r } }}{{1 - \beta ^r }}\left( {\Cl_2 \left( {2\arctan \left( {\sqrt { - \beta ^r } } \right)} \right) + \Cl_2 \left( {\pi  - 2\arctan \left( {\sqrt { - \beta ^r } } \right)} \right)} \right)\\
&\qquad\qquad + \left( {\frac{1}{{F_r \sqrt 5 }} + 1} \right)\frac{{\pi \sqrt { - \beta ^r } }}{{1 - \beta ^r }}\arctan \left( {\sqrt { - \beta ^r } } \right)\ln \left( {\sqrt { - \beta ^r } } \right).
\end{split}
\end{equation}
\end{theorem}
\begin{proof}
Use~\eqref{eq.ujhmyef} and~\eqref{eq.a40fggg} in~\eqref{eq.qvb6jur} to obtain~\eqref{eq.y4dqfp7}.
From~\eqref{eq.qvb6jur} and~\eqref{eq.a40qd9a} we have
\begin{equation*}
\begin{split}
&\int_0^\pi  {\frac{{x^2 \cos (3x)}}{{5F_r^2 - 4\cos^2 (2x)}}\,dx} \\
&\qquad = \left( {1 - \frac{1}{{F_r \sqrt 5 }}} \right)\frac{1}{2}\int_0^\pi  {\frac{{x^2 \cos x}}{{F_r \sqrt 5 - 2\cos (2x)}}\,dx}\\
&\qquad\qquad  - \left( {1 + \frac{1}{{F_r \sqrt 5 }}} \right)\frac{1}{2}\int_0^\pi  {\frac{{x^2 \cos x}}{{F_r \sqrt 5 + 2\cos (2x)}}\,dx} ,
\end{split}
\end{equation*}
and hence~\eqref{eq.gg6a1vx} upon using~\eqref{eq.ujhmyef} and~\eqref{eq.a40fggg}.
\end{proof}

\begin{lemma}
If $r$ is a positive integer, then
\begin{equation}\label{eq.bywzh6u}
\begin{split}
&\frac{1}{{2L_r }}\int_0^\pi {\frac{{x^2 \cos x}}{{L_r - 2\cos (2x)}}\,dx} + \frac{1}{{2L_r }}\int_0^\pi {\frac{{x^2 \cos x}}{{L_r + 2\cos (2x)}}\,dx} \\
&\qquad = \int_0^\pi  {\frac{{x^2 \cos x}}{{L_r^2 - 4\cos^2 (2x)}}\,dx} 
\end{split}
\end{equation}
and
\begin{equation}\label{eq.gaa3i0x}
\begin{split}
&\frac{1}{2}\int_0^\pi {\frac{{x^2 \cos x}}{{L_r - 2\cos (2x)}}\,dx} + \frac{1}{2}\int_0^\pi {\frac{{x^2 \cos x}}{{L_r + 2\cos (2x)}}\,dx} \\
&\qquad= \int_0^\pi {\frac{{x^2 \cos x}}{{L_r^2 - 4\cos ^2 2x}}\,dx} + \int_0^\pi {\frac{{x^2 \cos (3x)}}{{L_r^2 - 4\cos^2 (2x)}}\,dx}.
\end{split}
\end{equation}
\end{lemma}

\begin{theorem}
If $r$ is a positive even integer, then
\begin{equation}\label{eq.v11u6jr}
\begin{split}
&\int_0^\pi  {\frac{{x^2 \cos x}}{{L_r^2 - 4\cos^2 (2x)}}\,dx} \\
&\qquad=  - \frac{\pi }{{2L_r }}\frac{{\sqrt { \beta ^r } }}{{1 - \beta ^r }}\left( {\Li_2 \left( {\sqrt { \beta ^r } } \right) - \Li_2 \left( { - \sqrt { \beta ^r } } \right)} \right)\\
&\qquad\quad - \frac{\pi }{{2L_r }}\frac{{\sqrt { \beta ^r } }}{{1 + \beta ^r }}\left( {\Cl_2 \left( {2\arctan \left( {\sqrt { \beta ^r } } \right)} \right) + \Cl_2 \left( {\pi  - 2\arctan \left( {\sqrt { \beta ^r } } \right)} \right)} \right)\\
&\qquad\qquad - \frac{\pi }{L_r}\frac{{\sqrt { \beta ^r } }}{{1 + \beta ^r }}\arctan \left( {\sqrt { \beta ^r } } \right)\ln \left( {\sqrt { \beta ^r } } \right)
\end{split}
\end{equation}
and
\begin{equation}\label{eq.padu4yo}
\begin{split}
&\int_0^\pi  {\frac{{x^2 \cos (3x)}}{{L_r^2 - 4\cos^2 (2x)}}\,dx} \\
&\qquad = \left( {\frac{1}{L_r} - 1} \right)\frac{\pi }{2}\frac{{\sqrt { \beta ^r } }}{{1 - \beta ^r }}\left( {\Li_2 \left( {\sqrt { \beta ^r } } \right) - \Li_2 \left( { - \sqrt { \beta ^r } } \right)} \right)\\
&\qquad\quad+ \left( {\frac{1}{L_r} + 1} \right)\frac{\pi }{2}\frac{{\sqrt { \beta ^r } }}{{1 + \beta ^r }}\left( {\Cl_2 \left( {2\arctan \left( {\sqrt { \beta ^r } } \right)} \right) + \Cl_2 \left( {\pi  - 2\arctan \left( {\sqrt { \beta ^r } } \right)} \right)} \right)\\
&\qquad\qquad + \left( {\frac{1}{L_r} + 1} \right)\frac{{\pi \sqrt { \beta ^r } }}{{1 + \beta ^r }}\arctan \left( {\sqrt { \beta ^r } } \right)\ln \left( {\sqrt { \beta ^r } } \right).
\end{split}
\end{equation}
\end{theorem}
In particular,
\begin{equation*}
\begin{split}
\int_0^\pi  {\frac{{x^2 \cos x}}{{9 - 4\cos^2 (2x)}}\,dx} &= - \frac{\pi }{6}\left( {\frac{{\pi ^2 }}{6} - \frac{3}{2}\ln ^2 \alpha } \right) + \frac{\pi }{{6\sqrt 5 }}\arctan 2\ln \alpha \\
&\qquad - \frac{\pi }{{6\sqrt 5 }}\left( {\Cl_2 \left( {\arctan 2} \right) + \Cl_2 \left( {\pi  - \arctan 2} \right)} \right).
\end{split}
\end{equation*}

\end{document}